\documentclass[reqno,12pt]{amsart}
\usepackage{amsmath,epsfig}
\usepackage{amssymb}
\usepackage[ngerman,english]{babel}

\newtheoremstyle{rem}{1.3ex}{1.3ex}{\rmfamily}{}
{\itshape\rmfamily}{}{1.5ex}{}

\newtheorem{theorem}{Theorem}[section]
\newtheorem{lemma}[theorem]{Lemma}
\newtheorem{prop}[theorem] {Proposition}

\newtheorem{cor}[theorem] {Corollary}

\theoremstyle{definition}

\newtheorem{remark}[theorem] {Remark}

\renewcommand{\section}{\secdef\sct\sect}
\newcommand{\sct}[2][default]{\refstepcounter{section}
\setcounter{equation}{0}
\vspace{0.5cm}
\centerline{ \large
\scshape \arabic{section}.\ #1}
\vspace{0.3cm}}
\newcommand{\sect}[1]{
\vspace{0.5cm}
\centerline{\large\scshape #1}
\vspace{0.3cm}}

\renewcommand{\subsection}{\secdef \subsct\sbsect}
\newcommand{\subsct}[2][default]{\refstepcounter{subsection}
\nopagebreak
\vspace{0.5\baselineskip}
{\flushleft\bf \arabic{section}.\arabic{subsection}~\bf #1  }
\nopagebreak}
\newcommand{\sbsect}[1]{\vspace{0.1cm}\noindent
{\bf #1}\vspace{0.1cm}}

\newcommand{\G}{\mathcal{G}}
\newcommand{\R}     {\mathbb{R}}

\newcommand{\N}     {\mathbb{N}}
\renewcommand{\P}   {\mathbb{P}}

\newcommand{\E}     {\mathbb{E}}

\newcommand{\V}     {\mathbb{V}}

\def\1{{\mathchoice {1\mskip-4mu\mathrm l}
                    {1\mskip-4mu\mathrm l}
                    {1\mskip-4.5mu\mathrm l} {1\mskip-5mu\mathrm l}}}

\allowdisplaybreaks

\begin{document}

%\today
\title[On rates of convergence in the Curie-Weiss-Potts model]{\large
On rates of convergence \\\vspace{2mm}in the Curie-Weiss-Potts model \\\vspace{5mm}with an external field}

\author[Peter Eichelsbacher and Bastian Martschink ]{}
\maketitle
\thispagestyle{empty}
\vspace{0.2cm}

\centerline{\sc Peter Eichelsbacher\footnote{Ruhr-Universit\"at Bochum, Fakult\"at f\"ur Mathematik,
NA 3/66, D-44780 Bochum, Germany, {\tt peter.eichelsbacher@ruhr-uni-bochum.de  }}
and Bastian Martschink\footnote{Ruhr-Universit\"at Bochum, Fakult\"at f\"ur Mathematik,
NA 3/68, D-44780 Bochum, Germany, {\tt bastian.martschink@rub.de} \\Both authors have been supported by Deutsche Forschungsgemeinschaft via SFB/TR 12.}}

\vspace{2 cm}

%\centerline{\small{\version}}
%\vspace{1cm}

\begin{quote}
{\small {\bf Abstract:} 
In the present paper we obtain rates of convergence for the limit theorems of the density vector in the Curie-Weiss-Potts model via Stein's Method of exchangeable pairs. Our results include Kolmogorov bounds for multivariate normal approximation in the whole domain $\beta\geq 0$ and $h\geq 0$, 
where $\beta$ is the inverse temperature and $h$ an exterior field. In this model, the critical line $\beta = \beta_c(h)$ is explicitly known and corresponds to a first order transition. We include rates of convergence for non-Gaussian approximations 
at the extremity of the critical line of the model.}
\end{quote}

%\vfill



%\medskip\noindent
%Sumitted to EJP on Mai 22, 2009, final version accepted Mai 20, 2010.

%\newpage
\setcounter{section}{0}

\section{Introduction}

\subsection{The Curie-Weiss-Potts model}
The Curie-Weiss-Potts model is a mean field approximation of the well-known Potts model, a famous model in equilibrium statistical mechanics, see \cite{Kesten/Schonmann:1989} and \cite{Pearce/Griffiths:1980}.
It is defined in terms of a mean interaction averaged over all sites in the model, more precisely, by sequences of probability measures
of $n$ spin random variables that may occupy one of $q$ different states. For $q=2$ the model reduces to the simpler Curie-Weiss model. Probability limit theorems for the
Curie-Weiss-Potts model were first proven in \cite{Ellis/Wang:1990}. %One reason there is interest in this model is
%its explicit exhibition of a number of properties of real substances, such as multiple phase transitions, 
%metastable states and others. 
In comparison to the Curie-Weiss model it has a more intricate phase transition structure because 
for example at the critical inverse temperature it does not have a second-order phase transition like the Curie-Weiss model but a 
first-order phase transition. 
%In order to carry out the analysis of the model, detailed information about the structure of the set of
%canonical equilibrium macro-states is required.
The probability observing a configuration $\sigma \in \{1,\ldots,q\}^{n}$ in an exterior field $h$ equals
\begin{eqnarray}
P_{ \beta,h,n}( \sigma)=\frac{1}{Z_{ \beta,h,n}}\exp\left(\frac{ \beta}{2n}\sum\limits_{1\leq i\leq j\leq n}
\delta_{ \sigma_i,\sigma_j}+h\sum\limits_{i=1}^n\delta_{ \sigma_i,1}\right),\label{PBH}
\end{eqnarray}
where $\delta$ is the Kronecker symbol, $\beta:=T^{-1}$ is the inverse temperature and $Z_{ \beta,h,n}$ 
is the normalization constant known as the partition function. More precisely
\begin{eqnarray*}
Z_{ \beta,h,n}=\sum\limits_{\sigma \in \{1,\ldots,q\}^{n}}\exp\left(\frac{ \beta}{2n}\sum
\limits_{1\leq i\leq j\leq n}\delta_{ \sigma_i,\sigma_j}+h\sum\limits_{i=1}^n\delta_{ \sigma_i,1}\right).
\end{eqnarray*}
For $\beta$ small, the spin random variables are weakly dependent while for $\beta$ large they are strongly dependent.
It was shown in \cite{Wu:1982} that at $h=0$ the model undergoes a phase transition at the critical inverse temperature
\begin{eqnarray} \label{criticaltemp}
\beta_c =\begin{cases} q & ~,~\text{if~~} q \leq 2, \\ 2\, \frac{q-1}{q-2}\log(q-1) & ~,~\text{if~~} q> 2; \end{cases}
\end{eqnarray}
and that this transition is first order if $q>2$. 
Our interest is in the limit distribution of the empirical vector of the spin variables
\begin{equation} \label{N}
N=(N_1,\ldots, N_q)=\left(\sum\limits_{i=1}^n\delta_{\sigma_i,1},\ldots,\sum\limits_{i=1}^n\delta_{\sigma_i,q}\right),
\end{equation}
which counts the number of spins of each colour for a configuration $\sigma$. Note that the normalized empirical vector $L_n := N/n$ 
belongs to the set of probability vectors
\begin{eqnarray}\label{H}
\mathcal{H}=\{x\in\mathbb{R}^q:x_1+\cdots+x_q=1\text{ and } x_i\geq 0, \forall i\}.
\end{eqnarray}
For $q >2$ and $\beta < \beta_c$, $L_n$ satisfies the law of large numbers $P_{\beta,0,n} \bigl( L_n \in d\nu \bigr) 
\Rightarrow \delta_{\nu_0} (d \nu)$ as $n \to \infty$, where $\nu_0 = (1/q, \ldots, 1/q) \in \R^q$. For $\beta > \beta_c$ 
the law of large numbers breaks down and is replaced by the limit
$P_{\beta,0,n} \bigl( L_n \in d \nu \bigr) \Rightarrow \frac 1q \sum_{i=1}^q \delta_{\nu_i(\beta)} (d \nu)
$, where $\{\nu_i(\beta), i =1, \ldots, q\}$ are $q$ distinct probability vectors in $\R^q$, distinct from $\nu_0$. The {\it first-order phase
transition} is the fact that for $i=1, \ldots, q$ one has
$ %\begin{eqnarray*}
\lim_{\beta \to \beta_c^+} \nu_i(\beta) \not= \nu_0,
$ %\end{eqnarray*}
see \cite{Ellis/Wang:1990}. The case of {\it non-zero} external field $h \not= 0$ was considered in \cite{Biskup:2006} and it turned
out that the first-order phase transition remains on a critical line. The line was computed explicitly in \cite{Blanchard:2008}, see \eqref{criticalline}.

In the present work we obtain certain known probabilistic limit theorems for the Curie-Weiss-Potts model, especially for the empirical
vector of the spin variables $N$, {\it but at the same time} we present rates of convergence for all the limit theorems. We consider
the fluctuations of the empirical vector $N$ around its typical value outside the critical line
and we describe the fluctuations and rates of convergence at an extremity of the critical line. 
This extends previous results on the Curie-Weiss-Potts model with no external field \cite{Ellis/Wang:1990} as well as with external field
\cite{Gandolfo:2010}. The method of proof will be an application of Stein's method of so called exchangeable pairs in the case of 
multivariate normal approximation as well as the application of Stein's method in the case of non-Gaussian approximation. 
%Stein's method will be explained later. 
It might have been possible to obtain our results using the methods in \cite{Ellis/Wang:1990} and \cite{Gandolfo:2010},
using detailed analysis of the Hubbard-Stratonovich transform or applying Stirling's formula to the point probabilities, presumably which takes
considerably more work than the convergence result alone.

We turn to the description of the set of canonical equilibrium macro-states of the Curie-Weiss-Potts model, appealing the theory of large deviations. 
%These states are solutions 
%of an unconstrained minimization problem involving probability vectors on $\R^q$. The macro-states  describe equilibrium configurations of the model
%in the thermodynamic limit $n \to \infty$. For each $i$ the $i$'th component of an equilibrium macro-state gives the asymptotic relative frequency of
%spins taking the spin value $i$ with $i \in \{1, \ldots, q\}$.  We appeal to the theory of large deviations to define the canonical equilibrium macro-states.
Sanov's theorem states that with respect to the product measures $P_n(\omega) = 1/ q^n$ for $\omega \in \{1, \ldots, q \}^n$ the empirical vectors $L_n$
satisfy a large deviations principle (LDP)  on $\mathcal H$ with speed $n$ and rate function given by the relative entropy
$ %\begin{eqnarray*}
I(x) = \sum_{i=1}^q x_i \log (q x_i), \,\, x \in \mathcal H
$. %\end{eqnarray*}
We use the formal notation $P_n(L_n \in d x) \approx \exp ( -n I(x))$ (for a precise definition see \cite{Dembo/Zeitouni:LargeDeviations2ed}). The LDP for $L_n$ with respect to $P_{\beta,h,n}$
can be proven as in \cite{EllisHaven:2000}. Let
\begin{eqnarray*}
f_{\beta,h} (x) = \sum_{i=1}^q x_i \log (q x_i) - \frac{\beta}{2} \sum_{i=1}^q x_i^2 - h x_1, \,\, x \in {\mathcal H}.
\end{eqnarray*}
Then $P_{\beta,h,n} (L_n \in d x) \approx \exp ( -n J_{\beta,h}(x))$ with
\begin{eqnarray*}
J_{\beta,h}(x) :=f_{\beta,h} (x) - \inf_{x \in \mathcal H} f_{\beta,h} (x),
\end{eqnarray*}
see also \cite{CosteniucEllis:2005}. Now if $J_{\beta,h}(\nu) >0$, then $\nu$ has an exponentially small probability of being observed. Hence
the corresponding set of canonical equilibrium macro-states are naturally defined by
\begin{eqnarray*}
{\mathcal E}_{\beta,h} := \big\{ \nu \in \mathcal H :\,\, \nu \,\, \text{minimizes} \,\, f_{\beta,h}(\nu) \bigr\}.
\end{eqnarray*}
%Remark that the specific Gibbs free energy for the Curie-Weiss-Potts model is the quantity $\psi(\beta,h)$ defined by the limit
%\begin{eqnarray*}
%-\beta \psi(\beta,h) = \lim_{n \to \infty} \frac 1n \log Z_{\beta,h,n}.
%\end{eqnarray*}
%From the large deviations result it follows that
%\begin{eqnarray*}
%-\beta \psi(\beta,h) = - \inf_{x \in \mathcal H} f_{\beta,h}(x).
%\end{eqnarray*}
In the case $h=0$ and $q >2$, it is known since \cite{Wu:1982} (for  a detailed proof see \cite[Theorem 3.1]{CosteniucEllis:2005}) that ${\mathcal E}_{\beta,0}$ 
consists of one element for any $0 < \beta < \beta_c$, where $\beta_c$ is the critical inverse temperature given in \eqref{criticaltemp}.
For any $\beta > \beta_c$, the set consists of $q$ elements and at $\beta_c$ it consists of $q+1$ elements. 
In the case with an external field $h \geq 0$ the global minimizers of $f_{\beta,h}$ can be described as follows. In \cite{Blanchard:2008}
the following {\it critical line} was computed.
\begin{equation} \label{criticalline}
h_T :=\biggl\{ (\beta,h): 0 \leq h <h_0 \,\, \text{ and} \,\, h= \log(q-1)-\beta\frac{q-2}{2(q-1)} \biggr\},
\end{equation}
with extremities $(\beta_c,0)$ and $(\beta_0,h_0)$, where
\begin{eqnarray*}
\beta_0 =4 \frac{q-1}{q} \quad \text{and} \quad h_0=\log(q-1)-2\frac{q-2}{q}
\end{eqnarray*}
($(\beta_0, h_0)$ were already determined in \cite{Biskup:2006}).
Now consider the parametrization
\begin{eqnarray*}
x_z := \left(\frac{1+z}{2},\frac{1-z}{2(q-1)},\ldots,\frac{1-z}{2(q-1)}\right), \quad  z\in [-1,1].
\end{eqnarray*}

Depending on the parameters $(\beta, h)$ the function $f_{\beta,h}$ presents one or several global minimizers. The following
statement summarizes the results of \cite{Wu:1982}, \cite{CosteniucEllis:2005} in the case $h=0$ and of \cite{Blanchard:2008} for $h>0$.

\begin{theorem} \label{Minima}
Let $\beta,h\geq 0$.
\begin{enumerate}
\item If $h>0$ and $(\beta,h) \notin h_T$, the function $ f_{\beta,h}$ has a unique global minimum point in $\mathcal{H}$. 
This minimizer is analytic in $\beta$ and $h$ outside of $h_T\cup \{(\beta_0,h_0)\}$.
\item If $h>0$ and $(\beta,h)\in h_T$, the function $f_{\beta,h}$ has two global minimum points in $\mathcal{H}$. More precisely, 
for any $z\in (0,(q-2)/q)$, the two global minimum points of $f_{\beta_z,h_z}$ at
$$
\beta_z=2\frac{q-1}{zq}\log\left(\frac{1+z}{1-z}\right) \quad \text{and} \quad  h_z=\log(q-1)-\frac{q-2}{2(q-1)}\beta_z
$$
are the points $x_{\pm z}$.
\item If $h=0$ and $\beta<\beta_c$, the unique global minimum point of $f_{\beta,0}$ is $(1/q,\ldots,1/q)$.
\item If $h=0$ and $\beta>\beta_c$, there are $q$ global minimum points of $f_{\beta,0}$, which all equal $x_z$ up to a permutation of the 
coordinates for some $z\in ((q-2)/q,1)$.
\item If $h=0$ and $\beta=\beta_c$, there are $q+1$ global minimum points of $f_{\beta,0}$: the symmetric one $(1/q,\ldots,1/q)$ together with the permutations of 
$$
\left(\frac{q-1}{q},\frac{1}{q(q-1)},\ldots,\frac{1}{q(q-1)}\right).
$$
\end{enumerate}
\end{theorem}

Interesting enough, the very first results on probabilistic limit theorems (\cite{Ellis/Newman:1978} for the Curie-Weiss model and \cite{Ellis/Wang:1990}
for the Curie-Weiss-Potts model) used the structure of the global minimum points of another function $G_{\beta,h}$. %For the Curie-Weiss-Potts model with $h=0$ this function
%is given by
%\begin{eqnarray*}
%G_{\beta,0}(u) := \frac 12 \beta \langle u, u \rangle - \log \sum_{i=1}^q e^{\beta u_i}, \quad u \in \R^q.
%\end{eqnarray*}
%With convex duality one obtains the alternative representation of the specific Gibbs free energy given by
%\begin{eqnarray*}
%\beta \psi(\beta,0) = \min_{u \in \R^q} G_{\beta,0} (u) + \log q.
%\end{eqnarray*}
%Actually $f_{\beta,0}$ and $G_{\beta,0}$ have the same global minimum points, see \cite{Kesten/Schonmann:1989} for $\beta \not= \beta_c$.
%A proof of this result for any $\beta>0$ can be found in \cite[Theorem 3.1]{CosteniucEllis:2005}.
%The main reason to use $G_{\beta,0}$ instead of $f_{\beta,0}$ is the usefulness of a representation of the distribution of $L_n$ in terms 
%of $G_{\beta,h}$, called {\it Hubbard-Stratonovich transform} (see
%\cite[Lemma 3.2]{Ellis/Wang:1990} and the proof of Lemma \ref{EnMo} in this paper). This is a famous tool since the work of Ellis and Newman \cite{Ellis/%Newman:1978}.
For $\beta>0$ and $h$ real the global minimum points of $f_{\beta,h}$ coincide with the global minimum points of the function
\begin{equation} \label{Gfunktion}
G_{\beta,h}(u) := \frac 12 \beta \langle u, u \rangle - \log \bigl( \sum_{i=1}^q \exp ( \beta u_i + h \delta_{i,1}) \bigr), \quad u \in \R^q
\end{equation}
(for a proof see \cite[Theorem B.1]{Ellis/Wang:1992}; or apply a general result on minimum points of certain functions related
by convex duality, \cite[Theorem A.1]{CosteniucEllis:2005}, see also \cite{Wang:1994}).
Hence we know that all statements of Theorem \ref{Minima} hold true for $G_{\beta, h}$.

\begin{cor} \label{Mi}
The statements in Theorem \ref{Minima} for the global minimum points of $f_{\beta,h}$ hold true one to one for $G_{\beta,h}$, defined in \eqref{Gfunktion}. 
\end{cor}

%The {\it detour} first describing the canonical equilibrium macro-states of the Curie-Weiss-Potts model using large deviation theory
%and second using convex duality has the following reason.
 The main reason to consider $G_{\beta,h}$ instead of $f_{\beta,h}$ is the usefulness of a representation of the distribution of $L_n$ in terms 
of $G_{\beta,h}$, called {\it Hubbard-Stratonovich transform} (see
\cite[Lemma 3.2]{Ellis/Wang:1990} and the proof of Lemma \ref{EnMo} in this paper)
 Applying Stein's method we will also {\it meet} the function $G_{\beta,h}$
and the limit theorems and the proof of certain rates of convergence depend on the location of the global minimum points of $G_{\beta,h}$.
%(as in \cite{Ellis/Newman:1978}, \cite{Ellis/Wang:1990} and \cite{Ellis/Wang:1992} and a lot of other papers). 
For $h>0$ only
$f_{\beta, h}$ and its minimizers were completely characterized in the literature, see Theorem \ref{Minima}. 
%So we had to argue that we also know the phase diagram of $G_{\beta,h}$.

\subsection{Statement of the results}
Let us fix some notation. From now on we will write random vectors in $\mathbb{R}^d$ in the form $w=(w_1,\ldots,w_d)^t$, where $w_i$ are 
$\mathbb{R}$-valued variables for $i=1,\ldots,d$. If a matrix $\Sigma$ is symmetric, nonnegative definite, we denote by $\Sigma^{1/2}$ the unique 
symmetric, nonnegative definite square root of $\Sigma$. $\text{Id}$ denotes the identity matrix and from now on $Z$ will denote a random vector 
having standard multivariate normal distribution. The expectation with respect to the measure $P_{\beta,h,n}$ will be denoted by $\mathbb{E}:=\mathbb{E}_{P_{\beta,h,n}}$. 

Let $q>2$ for the whole paper. We first consider the issue of the fluctuations of the empirical vector $N$ defined in \eqref{N} around its typical value.
The case of the Curie-Weiss model ($q=2$) was considered in \cite{Griffiths/Simon:1973} and \cite{Ellis/Newman:1978} and a Berry-Esseen bound was proved in  \cite{Eichelsbacher/Loewe:2010} and independently in \cite{Chatterjee/Shao:2009}. The Curie-Weiss-Potts model was treated in \cite{Ellis/Wang:1990}
and for non-zero external field in \cite[Theorem 3.1]{Gandolfo:2010}. To the best of our knowledge rates of convergence were not considered.

We regard
\begin{align}\label{DefW}
W := \sqrt{n} \biggl( \frac{N}{n} - x_0 \biggr) = \sqrt{n} \bigl(L_n - x_0 \bigr).
\end{align}

\begin{theorem} \label{THUM}
Let $\beta>0$ and $h\geq 0$ with $(\beta,h) \neq (\beta_0,h_0)$. Assume that there is a unique minimizer $x_0$ of $G_{\beta,h}$. 
Let $W$ be defined in \eqref{DefW}.
If $Z$ has the $q$-dimensional standard normal distribution, we have for every three times differentiable function $g$,
$$
\big| \mathbb{E} g(W) - \mathbb{E} g \left( \Sigma^{1/2} Z\right) \big| \leq C \cdot n^{-1/2},
$$
for a constant $C$ (depending on $\beta$, $h$ and $q$) and $\Sigma := \mathbb{E}\left[W \, W^t \right]$. 
%Indeed we obtain that $C=\mathcal{O}\left(q^6\right)$.
\end{theorem}

Note that we compare the distribution of the rescaled vector $N$ with a multivariate normal distribution with covariance matrix $\E [ W \, W^t ]$.
It is an advantage of Stein's method that, for any fixed number of particles/spins $n$, we are able to compare the distribution of $W$ with a distribution 
with the same $n$-dependent covariance structure.\\ 
In order to state our next result we introduce conditions on the function classes $\mathcal{G}$ we consider. Following \cite{Rinott/Rotar:1996}, let $\Phi$ denote the standard normal distribution in $\mathbb{R}^q$. We define for $g:\mathbb{R}^q\rightarrow \mathbb{R}$
\begin{align}
g_{\delta}^+(x)&=\sup\bigl\{g(x+y):|y|\leq\delta\bigr\},\label{g+}\\
g_{\delta}^-(x)&=\inf\bigl\{g(x+y):|y|\leq\delta\bigr\},\label{g-}\\
\tilde g(x,\delta)&=g_{\delta}^+(x)-g_{\delta}^-(x)\label{gschl}.
\end{align}
Let $\mathcal{G}$ be a class of real measurable functions on $\mathbb{R}^q$ such that
\begin{enumerate}
\item The functions $g\in\mathcal{G}$ are uniformly bounded in absolute value by a constant, which we take to be 1 without loss of generality.
\item For any $q\times q$ matrix $A$ and any vector $b\in\mathbb{R}^q$, $g\bigl(Ax+b\bigr)\in\mathcal{G}$.
\item For any $\delta>0$ and any $g \in {\mathcal G}$, $g_{\delta}^+(x)$ and $g_{\delta}^-(x)$ are in $\mathcal{G}$.
\item For some constant $a=a(\mathcal{G},q)$, $\sup\limits_{g\in\mathcal{G}}\left\{\int\limits_{\mathbb{R}^q}\tilde g(x,\delta)\Phi(dx)\right\}\leq a\delta$. Obviously we may assume $a\geq 1$.
\end{enumerate}
Considering the one dimensional case, we notice that the collection of indicators of all half lines, and indicators of all intervals form classes in $\mathcal{G}$ 
that satisfy these conditions with $a=\sqrt{2/\pi}$ and $a=2\sqrt{2/\pi}$ respectively. This was shown for example in \cite{Rinott/Rotar:1996}. 
In dimension $q \geq 1$ the class of indicators of convex sets is known to be such a class. 
Using this notation we are able to present a result analogous to Theorem \ref{THUM} for our function classes $\mathcal{G}$.

\begin{theorem} \label{THUM2}
Let $\beta>0$ and $h\geq 0$ with $(\beta,h) \neq (\beta_0,h_0)$. Assume that there is a unique minimizer $x_0$ of $G_{\beta,h}$. 
Let $W$ and $Z$ be as in Theorem \ref{THUM}. Then, for all $g\in \mathcal{G}$, we have
$$
\big| \mathbb{E} g(W) - \mathbb{E} g \left( \Sigma^{1/2} Z\right) \big| \leq C \log(n)\cdot n^{-1/2},
$$
for a constant $C$ (depending on $\beta$, $h$ and $q$) and $\Sigma := \mathbb{E}\left[W \, W^t \right]$.
\end{theorem}

%Letting ${\mathcal G}$ be the collection of indicators of lower quadrants the distance above specializes to the Kolmogorov distance.

 When the function $G_{\beta,h}$ has several global minimizers, the empirical vector $N/n$ is close to either one or the other of these minima.
We determine the conditional fluctuations and a rate of convergence. $B(x^{(i)}, \epsilon)$ denotes the open ball of radius $\epsilon$
around a vector $x^{(i)}\in\R^q$.

\begin{theorem} \label{THMM}
Assume that $\beta,h \geq 0$ and that $G_{\beta,h}$ has multiple global minimum points $x^{(1)},\ldots, x^{(l)}$ with $l\in\{2,q,q+1\}$ (see Theorem \ref{Minima}, (2), (4) and (5)) and let 
$\epsilon>0$ be smaller than the distance between any two global minimizers of $G_{\beta,h}$. Furthermore, 
let 
\begin{align}\label{Wi}
W^{(i)}:=\sqrt{n}\left(\frac{N}{n}-x^{(i)}\right).
\end{align}
 Then, if $Z$ has the q-dimensional standard normal distribution, under the conditional measure
\begin{center}
$P_{\beta,h,n}\left(\cdot\mid \frac{N}{n}\in B(x^{(i)},\epsilon)\right)$,
\end{center}
we have for every three times differentiable function $g$,
\begin{center}
$\big| \mathbb{E}_i g(W^{(i)}) - \mathbb{E}_i g\left(\Sigma^{1/2} Z\right)  \big| \leq C\cdot n^{-1/2},$
\end{center}
for a constant C (depending on $\beta$, $h$ and $q$) and $\Sigma^{(i)}:= \mathbb{E}_i \left[ W^{(i)} \, (W^{(i)})^t \right]$, where $\mathbb{E}_i$ denotes the expectation with respect to the conditional probability.
\end{theorem}

We note that we can obtain a similar result as in Theorem \ref{THUM2} for the function class $\mathcal{G}$ in the case of several global minimizers. Finally we will take a look at the extremity $(\beta_0,h_0)$ of the critical line $h_T$. Given a vector $u\in\mathbb{R}^q$, 
we denote by $u^{\bot}$ the vector space made of all vectors orthogonal to $u$ in the Euclidean space $\mathbb{R}^q$. Consider the hyperplane
\begin{equation} \label{hyperM}
\mathcal{M}:=\bigg\{ x \in \mathbb{R}^q: \sum\limits_{i=1}^q x_i=0 \bigg\},
\end{equation}
which is parallel to $\mathcal H$ defined in \eqref{H}. The fluctuations belong to $\mathcal M$, since all global minimizers are in $\mathcal H$.
The following result extends \cite[Theorem 3.9]{Ellis/Newman:1978} which applies to the case of the Curie-Weiss model at the critical inverse temperature.
Recall that at $(\beta_0, h_0)$ the function $G_{\beta_0,h_0}$ has the unique minimizer $x=(1/2, 1/2(q-1), \ldots, 1/2(q-1)) \in \R^q$.
Now we take $u=(1-q,1,\ldots, 1)\in \mathcal{M}\subset\R^q$ and define a real valued random variable $T$ and a random vector $V\in\mathcal{M}\cap u^{\bot}$
such that
\begin{equation} \label{defTV}
N= n\, x+ n^{3/4} \, T\, u +n^{1/2} \,V.
\end{equation}
Since $N- n \, x \in {\mathcal M}$, the implicit definition of $T$ and $V$ presents a partition into a vector in (the subspace spanned by) $u$ and $u^{\bot}$.
The main interest is the limiting behavior of $T$. The new scaling of $W$ is given by
$$
\frac{N_j - n \frac{1}{2 (q-1)}}{n^{3/4}} = T + V_j/n^{1/4}, \quad j=2, \ldots, q,
$$
and its possible limit we observe is reminiscent to \cite{Ellis/Newman:1978}, see also \cite{CosteniucEllis:2007}. The following
theorem gives a Kolmogorov bound for Theorem 3.7 in \cite{Gandolfo:2010}.

\begin{theorem} \label{TT}
For $(\beta,h)=(\beta_0,h_0)$ let $x=(1/2,1/2(q-1),\ldots,1/2(q-1))$ be the unique minimizer of $G_{\beta_0,h_0}$ and 
$u=(1-q,1,\ldots,1)$. Furthermore, let $Z_{q,T}$ be a random variable distributed according to the probability measure on $\mathbb{R}$ with the density
$$
f_{q,T}(t):= f_{q,T,n} := C \cdot \exp \left( - \frac{1}{4 \, \E(T^4)} t^4\right),
$$
where $T$ is defined in \eqref{defTV}.
Then we obtain for any uniformly 1-Lipschitz function $g:\mathbb{R}\to \mathbb{R}$ that
$$
\big| \mathbb{E} g(T) - \mathbb{E} g\left(Z_{q,T}\right) \big| \leq C \cdot n^{-1/4}.
$$
Moreover we obtain 
$$
\sup_{t \in \R} \big| \P \bigl( T \leq t \bigr) - F_{q,T}(t) \big| \leq C \cdot n^{-1/4} \quad \text{(bound for the Kolmogorov-distance)},
$$
where $F_{q,T}$ denotes the distribution function of $f_{q,T}$. The constants $C$ depend on $q$.
\end{theorem}

\begin{remark} As we will see in the proof of Theorem \ref{TT}, the density $f_{q,T}$ has the form 
$$
\exp \left( - \frac{4(q-1)^4}{3 \, \E(T \psi(T))} t^4 \right)
$$
(up to a constant) with a function $\psi$ such that $\E(T \psi(T))= \frac{16(q-1)^4}{3} \E(T^4)$.
From \cite[Theorem 3.7]{Gandolfo:2010} we know that $T$ converges in distribution to the probability measure on $\R$
proportional to
$$
g_{q}(t) = \exp \left( - \frac{4(q-1)^4}{3} t^4\right).
$$
Hence we conclude that $\lim_{n \to \infty} f_{q,T,n} = g_q$ point-wise and therefore $\frac{16(q-1)^4}{3} \, \E\left[T^4\right] \to 1$.
Note that the rate of convergence of Theorem \ref{TT} also holds when we compare the distribution of $T$ with
the law on $\R$ with density proportional to $g_q$. 
\end{remark}

Additionally we get a theorem for the random vector $V$, improving Theorem 3.7 in \cite{Gandolfo:2010}.

\begin{theorem} \label{TV} 
Let $V$ be defined as in \eqref{defTV}. For $(\beta,h)=(\beta_0,h_0)$ and 
$\Sigma:= \mathbb{E}[V \, V^t]$
we have that for every three times differentiable function $g$,
$$
\big| \mathbb{E} g(V) - \mathbb{E} g\left(\Sigma^{1/2} Z\right) \big| \leq C \cdot n^{-1/4},
$$
where $C$ is a constant depending on $q$.
\end{theorem}

%\begin{remark}
%The proofs of Theorem \ref{TT} and Theorem \ref{TV} employ a fourth-order Taylor expansion of $G_{\beta_0,h_0}$, see \eqref{finaleins} and \eqref{finalzwei} in the Appendix.
%Without a doubt, the first and the third term in \eqref{finaleins}, as well as in \eqref{finalzwei}, gives the order ${\mathcal O}(n^{-1/4})$.
%\end{remark}

%\newpage

%In Section 2 of the present paper, the limit theorems and the rates of convergence are stated. They include a central limit theorem and a bound on the distance to a multivariate normal distribution
%for $L_n$ outside the critical line. When $G_{\beta,h}$ has several global minimizers, that is when $(\beta,h) \in h_T$ or $\beta \geq \beta_c$ and $h=0$,
%the empirical vector $L_n$ is close to either one or the other of the minimizers. In this case we determine a central limit theorem with conditioning for $L_n$. 
%Next we describe the fluctuations at the extremity $(\beta_0, h_0)$ of the critical
%line, again combined with a rate of convergence. 
In Section 2 we give a short introduction in Stein's method and state an abstract nonsingular multivariate normal approximation theorem for smooth test functions from \cite{ReinertRoellin:2009}. Moreover we present a new bound for non smooth test functions for bounded random vectors $W$ under exchangeability.
Finally we state an abstract non-Gaussian univariate approximation theorem for the Kolmogorov-distance from \cite{Eichelsbacher/Loewe:2010}. 
Section 3 contains some auxiliary results which will be necessary for the proofs given in Section 4.

\section{Stein's method}
Starting with a bound for the distance between univariate random variables and the normal distribution Stein's method was first 
published in \cite{Stein:1972} (1972). Being particularly powerful in the presence of both local dependence and weak global dependence his 
method has proven to be very successful.
In \cite{Stein:1986} Stein explained his exchangeable pair approach in detail. At the heart of the method is a coupling of a random variable $W$ with 
another random variable $W'$ such that $(W,W')$ is {\it exchangeable}, i.e. their joint distribution is symmetric. 
Stein proved further on that a measure of proximity of W to normality may be provided by the exchangeable pair if $W'-W$ is sufficiently small. 
He assumed the property that there is a number $\lambda>0$ such that the expectation of $W'-W$ with respect to W satisfies
$$
\mathbb{E}[W'-W\mid W]=-\lambda W.
$$
Heuristically, this condition can be understood as a linear regression condition. If $(W,W')$ were bivariate normal with correlation $\varrho$, then
$\E (W' |W) = \varrho \, W$ and the condition would be satisfied with $\lambda = 1 - \varrho$.
While the exchangeable pair approach has proven successful also in non-normal contexts (see \cite{ChatterjeeDiaconis:2005},\cite{Chatterjee/Shao:2009} and \cite{Eichelsbacher/Loewe:2010}) 
it remained restricted to the one-dimensional setting for a long time. However in \cite{ChatterjeeMeckes:2008} and \cite{ReinertRoellin:2009} this issue was finally addressed. For an exchangeable pair $(W, W')$ of $\R^d$-valued random vectors the linear regression heuristic leads to a new condition due to \cite{ReinertRoellin:2009} given by
\begin{equation} \label{regressioncond}
\mathbb{E}[W'-W\mid W]=-\Lambda W+R
\end{equation}
for an invertible $d\times d$ matrix $\Lambda$ and a remainder term $R=R(W)$. 
Different exchangeable pairs, obviously, will yield different $\Lambda$ and $R$. 

Let us fix some more notations. The transpose of the inverse of a matrix will be presented in the form $A^{-t}:=(A^{-1})^t$. 
Furthermore we will need the supremum norm, denoted by $\parallel \cdot\parallel$ for both functions and matrices. 
For derivatives of smooth functions $f: \R^d \to \R$, we use the notation $\nabla$ for the gradient operator.
For a function $f:\mathbb{R}^d\to \mathbb{R}$, we abbreviate 
\begin{equation} \label{speznorm}
|f|_1 := \sup\limits_{i}\parallel\frac{\partial}{\partial x_i}f\parallel, \quad |f|_2:=
\sup\limits_{i,j}\parallel\frac{\partial^2}{\partial x_i\partial x_j}f\parallel,
\end{equation}
and so on, if these derivatives exist. 

The method of Stein is based on the characterization of the normal distribution that $Y \in \mathbb{R}^d$, $d\in \mathbb{N}$, is a centered multivariate normal 
vector with covariance matrix $\Sigma$ if and only if
$$ %\begin{equation} \label{Steinchar}
\mathbb{E}\left[ \nabla^t \Sigma \nabla f(Y)- Y^t \nabla f(Y) \right]=0 \quad \text{for all smooth} \quad  f:\mathbb{R}^d \to \mathbb{R}.
$$ %\end{equation}
It is well known, see \cite{Barbour:1990} and \cite{Goetze:1991}, that for any $g:\mathbb{R}^d\to\mathbb{R}$ being differentiable with bounded first derivatives,
if $\Sigma\in\mathbb{R}^{d\times d}$ is symmetric and positive definite, there is a solution $f:\mathbb{R}^d \to \mathbb{R}$ to the equation
\begin{equation} \label{Steinequation}
\nabla^t \Sigma \nabla f(w) - w^t \nabla f(w)= g(w)- \mathbb{E} g \left(\Sigma^{1/2} Z\right),
\end{equation}
which holds for every $w\in\mathbb{R}^d$. If, in addition, $g$ is $n$ times differentiable, there is a solution $f$ which is also $n$ times differentiable
and one has for every $k=1, \ldots, n$ the bound
$ %\begin{equation*} %\label{Steinbound}
\bigg| \frac{\partial^k f(w)}{\prod_{j=1}^k \partial w_{i_j}} \bigg| \leq \frac 1k \bigg| \frac{\partial^k g(w)}{\prod_{j=1}^k \partial w_{i_j}} \bigg|
$ %\end{equation*}
for every $w \in \R^d$. We will apply Theorem 2.1 in \cite{ReinertRoellin:2009}.

\begin{theorem} \label{RR}(Reinert, R\"ollin: 2009)\\
Assume that $(W,W')$ is an exchangeable pair of $\mathbb{R}^d$-valued random vectors such that
$$
\mathbb{E}[W]=0, \quad \mathbb{E} [W \, W^t] =\Sigma,
$$
with $\Sigma \in \mathbb{R}^{d\times d}$ symmetric and positive definite. If $(W,W')$ satisfies \eqref{regressioncond}
for an invertible matrix $\Lambda$ and a $\sigma(W)$-measurable random vector $R$ and if $Z$ has d-dimensional standard normal distribution, 
we have for every three times differentiable function $g$,
\begin{equation} \label{mainboundRR}
 \big| \mathbb{E} g(W) - \mathbb{E} g\left(\Sigma^{1/2} Z\right) \big| \leq \frac{|g|_2}{4}A+\frac{|g|_3}{12}B+\left(|g|_1+\frac{1}{2}d\parallel\Sigma\parallel^{1/2}
|g|_2\right)C,
\end{equation}
where, with $\lambda^{(i)}:= \sum\limits_{m=1}^d \big| (\Lambda^{-1})_{m,i} \big|$,
\begin{eqnarray}
A &= &\sum\limits_{i,j=1}^d \lambda^{(i)} \sqrt{\mathbb{V} \left[\mathbb{E}[(W'_i-W_i)(W'_j-W_j)\mid W]\right]}, \nonumber \\ %\nonumber
B&=&\sum\limits_{i,j,k=1}^d\lambda^{(i)}\mathbb{E} | (W'_i-W_i)(W'_j-W_j)(W'_k-W_k) |,\\
C&=&\sum\limits_{i=1}^d\lambda^{(i)}\sqrt{\mathbb{E}\left[R_i^2\right]}. \nonumber
\end{eqnarray}
\end{theorem}

The advantage of Stein's method is that the bounds to a multivariate normal distribution reduce to the 
computation of, or bounds on, low order moments, here bounds on the absolute third moments, on a conditional variance and on the variance 
of the remainder term. Such variance computations may be difficult, but we will get rates of convergence at the same time.
In the same context as in \cite{ReinertRoellin:2009} we can show the following theorem, presenting bounds for non-smooth test functions (improving Corollary 3.1 in \cite{ReinertRoellin:2009}). 
Our development differs from Reinert and R\"ollin, as we use the relationship to the bounds in \cite{Rinott/Rotar:1997}. 
%We obtain a bound of 
%order $\log(n)\cdot n^{-1/2}$ assuming some boundedness, improving Corollary 3.1 in \cite{ReinertRoellin:2009}.

\begin{theorem}\label{Kolmogorov}~\\
Let $(W,W')$ be an exchangeable pair with $\E[W]=0$ and $\E[WW^t]=\Sigma$ with $\Sigma\in\R^{d\times d}$ symmetric and positive definite. Again we assume that $(W,W')$ satisfies \eqref{regressioncond} for an invertible matrix $\Lambda$ and a $\sigma(W)$-measurable random vector $R$ and additionally, for $i\in\{1,\ldots,d\}$, $|W_i'-W_i|\leq A$. Then,
\begin{eqnarray*}
\sup_{g\in\G} |\E g(W)-\E g(\Sigma^{1/2}Z)| &\leq& C \bigl[
\log(t^{-1})A_1+\left(\log(t^{-1})\Vert\Sigma\Vert^{1/2}+1\right)A_2\\
&&+\biggl(1+\log(t^{-1})\sum\limits_{i=1}^d\E|W_i|+a\biggr)A^3A_3
+aA\bigr],
\end{eqnarray*}
where
\begin{eqnarray*}
A_1 &= & \sum\limits_{i,j=1}^d|(\Lambda^{-1})_{j,i}|\sqrt{\V\bigl[\E[(W'_i-W_i)(W'_j-W_j)| W]\bigr]},\\
A_2&=&\sum\limits_{i,j=1}^d|(\Lambda^{-1})_{j,i}|\sqrt{\E\left[R_i^2\right]}, \qquad
A_3=\sum\limits_{i=1}^d\max\limits_{j\in\{1,\ldots,d\}}|(\Lambda^{-1})_{j,i}|,
\end{eqnarray*} 
C denotes a constant that depends on $d$, $\sqrt{t}=2CA^3A_3$ and $a>1$ is taken from the conditions on $\G$, defined after Theorem \ref{THUM}.
\end{theorem}
\begin{proof}
First we assume that $\Sigma=\text{Id}$. Throughout the proof we write $C$ for universal constants, 
not necessarily the same at each occurrence. We consider the multivariate Stein equation deduced from \eqref{Steinequation} with $\Sigma=\text{Id}$ given by
\begin{align}\label{MVStein}
\nabla^t \nabla f(w)-w^t\cdot \nabla f(w)&=g(w)-\E[g(Z)].
\end{align}
For $g \in \G$ define the following smoothing
\begin{align}\label{gt}
g_t(x)=\int\limits_{\R^d}g\left(\sqrt{t}z+\sqrt{1-t}x\right)\Phi(z)dz.
\end{align}
For $g_t$, \eqref{MVStein} is solved by the function
$
f_t(x)=-\frac{1}{2}\int\limits_{t}^{1} \bigl( g_s(x)-\mathbb{E}[g(Z)] \bigr) \frac{ds}{1-s},
$
see \cite{Goetze:1991}. Again by \cite{Goetze:1991} (see also \cite{Bhattacharya/Holmes:2010}), we have that for $|g|\leq 1$, there exists a constant $C$, depending
only on the dimension $d$,  such that in the notation of \eqref{speznorm}
\begin{align}\label{GoeAb}
|f_t|_1 &\leq C,\\\label{GoeAb2}
|f_t|_2 & \leq C\log(t^{-1}).
\end{align}
According to \cite{Goetze:1991} (see also \cite{ReinertRoellin:2009}, Lemma A.1) there is also a constant $C>0$, depending on $d$, such that for all $t\in(0,1)$
\begin{align}\label{Abnonsm}
\sup\bigl\{\left|\mathbb{E}g(W)-\mathbb{E}g(Z)\right|:g\in\mathcal{G}\bigr\}&\leq C\cdot \delta_t+a\sqrt{t},
\end{align}
where $a>1$ is the constant appearing in (4) in the definition of $\mathcal{G}$ and
\begin{align}\nonumber
\delta_t:=\sup\bigl\{\left|\mathbb{E}g_t(W)-\mathbb{E}g_t(Z)\right|:g\in\mathcal{G}\bigr\}.
\end{align}
 Thus, it remains to estimate $\delta_t$. We see that, for an exchangeable pair  we have
 %and the real-valued, anti-symmetric function
 %$
 %F(w',w):=\frac{1}{2}(w'-w)^t\Lambda^{-t}(\nabla f_t(w')+\nabla f_t(w))
 %$
 %\end{align*}
 %with $w,w'\in\R^d$, we can apply \eqref{ASFunction} and thus 
%\begin{align*}
$
0=\frac{1}{2}\E\bigl[(W'-W)^t\Lambda^{-t}(\nabla f_t(W')+\nabla f_t(W))\bigr].
$
%\end{align*}
After adding and subtracting the same expression we are able to use the linear regression condition on the expression $(W'-W)$ in the expectation that yields
\begin{align*}
0&=\E\bigl[(W'-W)^t\Lambda^{-t}\nabla f_t(W)\bigr]+\frac{1}{2}\E\bigl[(W'-W)^t\Lambda^{-t}(\nabla f_t(W')-\nabla f_t(W))\bigr]\\
&=\E\bigl[R^t\Lambda^{-t}\nabla f_t(W)\bigr]-\E\bigl[W^t\nabla f_t(W)\bigr]+\frac{1}{2}\E\bigl[(W'-W)^t\Lambda^{-t}(\nabla f_t(W')-\nabla f_t(W))\bigr].
\end{align*}
Abbreviating $f^{(1)}_j:=\frac{\partial}{\partial x_j}f_t$ and $f^{(2)}_{i,j}:=\frac{\partial^2}{\partial x_j\partial x_i}f_t$, etc. for a function $f_t$, we obtain 
\begin{align}\nonumber
\mathbb{E}\bigl[W^t\nabla f_t(W)\bigr]&=\frac{1}{2}\mathbb{E}\bigl[(W'-W)^t\Lambda^{-t}(\nabla f_t(W')-\nabla f_t(W))\bigr]+\mathbb{E}\bigl[R^t\Lambda^{-t}\nabla f_t(W)\bigr]\\\nonumber
&=\frac{1}{2}\sum\limits_{i,j=1}^d(\Lambda^{-1})_{j,i}\mathbb{E}\bigl[(W_i'-W_i)(f^{(1)}_j(W')-f^{(1)}_j(W))\bigr]\\\nonumber
&~~+\sum\limits_{i,j=1}^d(\Lambda^{-1})_{j,i}\mathbb{E}\bigl[R_if^{(1)}_j(W)\bigr].
\end{align}
Hence,\\
$\mathbb{E}g_t(W)-\mathbb{E}g_t(Z)$
\begin{align}\nonumber
&=\sum\limits_{i=1}^d\mathbb{E}\bigl[f^{(2)}_{i,i}(W)\bigr]-\frac{1}{2}\sum\limits_{i,j=1}^d(\Lambda^{-1})_{j,i}\mathbb{E}\bigl[(W_i'-W_i)(f^{(1)}_j(W')-f^{(1)}_j(W))\bigr]\\\nonumber
&~~-\sum\limits_{i,j=1}^d(\Lambda^{-1})_{j,i}\mathbb{E}\bigl[R_if^{(1)}_j(W)\bigr]\\\nonumber
&=\frac{1}{2}\mathbb{E}\biggl[2\sum\limits_{i=1}^df^{(2)}_{i,i}(W)-2\sum\limits_{i,j=1}^d(\Lambda^{-1})_{j,i}R_iW_jf^{(2)}_{j,j}(W)\\\nonumber
&~~-\sum\limits_{i,j=1}^d(\Lambda^{-1})_{j,i}(W'_i-W_i)(W'_j-W_j)f^{(2)}_{j,j}(W)\biggr]\\\nonumber
&~~+\mathbb{E}\biggl[\sum\limits_{i,j=1}^d(\Lambda^{-1})_{j,i}R_iW_jf^{(2)}_{j,j}(W)-\sum\limits_{i,j=1}^d(\Lambda^{-1})_{j,i}R_if^{(1)}_j(W)\biggr]\\\nonumber
&~~-\frac{1}{2}\mathbb{E}\biggl[\sum\limits_{i,j=1}^d(\Lambda^{-1})_{j,i}(W_i'-W_i)(f^{(1)}_j(W')-f^{(1)}_j(W))\\\nonumber
&~~-\sum\limits_{i,j=1}^d(\Lambda^{-1})_{j,i}(W'_i-W_i)(W'_j-W_j)f^{(2)}_{j,j}(W)\biggr]\\\nonumber
&=:J_1+J_2+J_3.
\end{align}
Using \eqref{GoeAb2} and the fact that $\mathbb{E}\bigl[(W'-W)(W'-W)^t\bigr]=2\Lambda^t-2\mathbb{E}\bigl[WR^t\bigr]$ we have
\begin{align}\nonumber
|J_1|&=\bigg|\frac{1}{2}\mathbb{E}\bigl[2\sum\limits_{i=1}^df^{(2)}_{i,i}(W)-2\sum\limits_{i,j=1}^d(\Lambda^{-1})_{j,i}R_iW_jf^{(2)}_{j,j}(W)\\\nonumber
&~~-\sum\limits_{i,j=1}^d(\Lambda^{-1})_{j,i}(W'_i-W_i)(W'_j-W_j)f^{(2)}_{j,j}(W)\bigl]\bigg|\\\nonumber
&\leq C\log(t^{-1})\sum\limits_{i,j=1}^d|(\Lambda^{-1})_{j,i}|\mathbb{E}\bigl|2(\Lambda)_{j,i}-2R_iW_j-\mathbb{E}\bigl[(W'_i-W_i)(W'_j-W_j)| W\bigr]\bigr|\\\nonumber
&\leq C\log(t^{-1})\sum\limits_{i,j=1}^d|(\Lambda^{-1})_{j,i}|\sqrt{\mathbb{E}\bigl[\left(2(\Lambda)_{j,i}-2R_iW_j-\mathbb{E}\bigl[(W'_i-W_i)(W'_j-W_j)| W\bigr]\right)^2\bigr]}\\ \nonumber
&=C\log(t^{-1})\sum\limits_{i,j=1}^d|(\Lambda^{-1})_{j,i}|\sqrt{\mathbb{V}\left[\mathbb{E}[(W'_i-W_i)(W'_j-W_j)| W]\right]}.
\end{align}
Additionally, again using \eqref{GoeAb} and \eqref{GoeAb2} and the fact that $\mathbb{E}\bigl[WW^t\bigr]=\text{Id}$, we have
\begin{align}\nonumber
|J_2|&=\bigg|\mathbb{E}\bigl[\sum\limits_{i,j=1}^d(\Lambda^{-1})_{j,i}R_iW_jf^{(2)}_{j,j}(W)-\sum\limits_{i,j=1}^d(\Lambda^{-1})_{j,i}R_if^{(1)}_j(W)\bigr]\bigg|\\\nonumber
&\leq C\log(t^{-1})\sum\limits_{i,j=1}^d|(\Lambda^{-1})_{j,i}| \, \mathbb{E}\left| R_iW_j\right| + C\sum\limits_{i,j=1}^d|(\Lambda^{-1})_{j,i}| \, \mathbb{E}\left|R_i\right|\\\nonumber
&\leq C( \log(t^{-1}) +1) \sum\limits_{i,j=1}^d|(\Lambda^{-1})_{j,i}|\sqrt{\mathbb{E} [R_i^2]}.
\end{align}
The estimation of $J_3$ is a bit more involved. With $\lambda_i=\max\limits_{j\in\{1,\ldots,d\}}|(\Lambda^{-1})_{j,i}|$ we have
\begin{align}\nonumber
|2J_3|&=\bigg|\sum\limits_{i,j=1}^d\mathbb{E}\biggl[(\Lambda^{-1})_{j,i}(W_i'-W_i)(f^{(1)}_j(W')-f^{(1)}_j(W))\\\nonumber
&~~-(\Lambda^{-1})_{j,i}(W'_i-W_i)(W'_j-W_j)f^{(2)}_{j,j}(W)\biggr]\bigg|\\\nonumber
&\leq \sum\limits_{i=1}^d \lambda_i \biggl| \mathbb{E} \biggl[ (W_i'-W_i) \sum_{j=1}^d (f^{(1)}_j(W')-f^{(1)}_j(W)-(W'_j-W_j)f^{(2)}_{j,j}(W))\biggr]\biggr|\\\nonumber
&\leq \sum\limits_{i=1}^d\lambda_i \bigg|\mathbb{E}\bigl[A^3\sum\limits_{j,k=1}^d\int\limits_{0}^1(1-\tau)f^{(3)}_{j,j,k}(W'-\tau(W'-W))d\tau\bigr] \bigg|.
\end{align}
We abbreviate $M:=W'-\tau(W'-W)$. It is important to notice that we can use \eqref{MVStein} to obtain for $w\in\mathbb{R}^d$
\begin{center}
$\sum\limits_{j,k=1}^df^{(3)}_{j,j,k}(w)+\sum\limits_{j,k=1}^dw_jf^{(2)}_{j,k}(w)+\sum\limits_{k=1}^df^{(1)}_{k}(w)=\sum\limits_{k=1}^dg^{(1)}_{k}(w)$.
\end{center}
Additionally we notice that
$ \int\limits_0^1(1-\tau)d\tau =\frac{1}{2}$ and $ 
\bigg| \int\limits_0^1(1-\tau)M_id\tau \bigg|  \leq|W'_i|+|W_i|$.
Thus,
\begin{align}\nonumber
|2J_3|&\leq \sum\limits_{i=1}^d\lambda_iA^3 \bigg| \mathbb{E} \biggl[ \int\limits_0^1(1-\tau) \biggl( -\sum\limits_{k=1}^df^{(1)}_{k}(M)-\sum\limits_{j,k=1}^dM_jf^{(2)}_{j,k}(M)
+\sum\limits_{k=1}^dg^{(1)}_{k}(M) \biggr) d\tau \biggr] \bigg|\\ \nonumber
&\leq \sum\limits_{i=1}^d\lambda_i\bigg( A^3C+A^3\sum\limits_{j=1}^dC\log(t^{-1})\mathbb{E}\bigl[|W'_j|+|W_j|\bigr]+ \bigg|A^3\sum\limits_{k=1}^d\mathbb{E}\bigl[\int\limits_0^1(1-\tau)g^{(1)}_{k}(M)d\tau\bigr]\bigg|\bigg)\\\nonumber
&\leq \sum\limits_{i=1}^d\lambda_i\biggl( A^3C\bigl(1+\log(t^{-1})\sum\limits_{j=1}^d\mathbb{E}\bigl|W_j\bigr|\bigr)+T\biggr)
\end{align}
with
\begin{align*}
T:=\bigg|A^3\sum\limits_{k=1}^d\mathbb{E}\bigl[\int\limits_0^1(1-\tau)g^{(1)}_{k}(M)d\tau\bigr]\bigg|.
\end{align*}
By partial integration we obtain
$$
\sum\limits_{k=1}^dg^{(1)}_{k}(w)=\sum\limits_{k=1}^d\frac{\sqrt{1-t}}{\sqrt{t}}\int\limits_{\mathbb{R}^d}g\left(\sqrt{t}z+\sqrt{1-t}w\right)\Phi^{(1)}_{k}(z)dz.
$$
Keeping in mind that $\sum\limits_{k=1}^d\int\limits_{\mathbb{R}^d}\Phi^{(1)}_{k}(z)dz=0$ and using the definitions \eqref{g+}, \eqref{g-} and \eqref{gschl} we obtain
\begin{align}\nonumber
T&\leq \frac{\sqrt{1-t}}{\sqrt{t}}A^3\bigg|\sum\limits_{k=1}^d\mathbb{E}\bigl[\int\limits_0^1\int\limits_{\mathbb{R}^d}(1-\tau)g(\sqrt{t}z+\sqrt{1-t}M)\Phi^{(1)}_{k}(z)dzd\tau\bigr]\bigg|\\\nonumber
&=\frac{\sqrt{1-t}}{\sqrt{t}}A^3\bigg|\sum\limits_{k=1}^d\mathbb{E}\bigl[\int\limits_0^1\int\limits_{\mathbb{R}^d}(1-\tau)[g(\sqrt{t}z+\sqrt{1-t}M)-g(\sqrt{1-t}M)]\Phi^{(1)}_{k}(z)dzd\tau\bigr]\bigg|\\\nonumber
&\leq \frac{\sqrt{1-t}}{2\sqrt{t}}A^3\sum\limits_{k=1}^d\mathbb{E}\bigl[\int\limits_{\mathbb{R}^d}[g_{\sqrt{1-t}A+\sqrt{t}|z|}^+(\sqrt{1-t}W)-g_{\sqrt{1-t}A+\sqrt{t}|z|}^-(\sqrt{1-t}W)]\bigl|\Phi^{(1)}_{k}(z)\bigr|dz\bigr]\\\nonumber
&\leq \frac{\sqrt{1-t}}{2\sqrt{t}}A^3\sum\limits_{k=1}^d\mathbb{E}\bigl[\int\limits_{\mathbb{R}^d}\underbrace{\tilde g(\sqrt{1-t}W;\sqrt{1-t}A+\sqrt{t}|z|)}_{=:\tilde g(W,A,t,z)}\bigl|\Phi^{(1)}_{k}(z)\bigr|dz\bigr]\\\nonumber
&\leq \frac{\sqrt{1-t}}{2\sqrt{t}}A^3\sum\limits_{k=1}^d\mathbb{E}\bigl[\int\limits_{\mathbb{R}^d}[\tilde g(W,A,t,z)-\tilde g(Z,A,t,z)]\bigl|\Phi^{(1)}_{k}(z)\bigr|dz\bigr]\\\nonumber
&~~+\frac{\sqrt{1-t}}{2\sqrt{t}}A^3\sum\limits_{k=1}^d\mathbb{E}\bigl[\int\limits_{\mathbb{R}^d}\tilde g(Z,A,t,z)\bigl|\Phi^{(1)}_{k}(z)\bigr|dz\bigr]\\\nonumber
&=:B_1+B_2.
\end{align}
With $\delta:=\sup\bigl\{\left|\mathbb{E}[g(W)]-\mathbb{E}[g(Z)]\right|:g\in\mathcal{G}\bigr\}$ we obtain
$$
B_1\leq C\frac{\sqrt{1-t}}{2\sqrt{t}}A^3\cdot \delta\leq \frac{C\delta}{\sqrt{t}}A^3.
$$
Furthermore by using the conditions established for the function class $\mathcal{G}$
\begin{align}\nonumber
B_2&\leq\frac{\sqrt{1-t}}{2\sqrt{t}}A^3\cdot a\left(\sqrt{1-t}A+\sum\limits_{k=1}^d\int\limits_{\mathbb{R}^d}\sqrt{t}\bigl|z\bigr|\bigl|\Phi^{(1)}_{k}(z)\bigr|dz\right)\\\nonumber
&=a\frac{1-t}{2\sqrt{t}}A^4+aC\frac{\sqrt{1-t}}{2}A^3\\\nonumber
&\leq aCA^3\left(\frac{A}{\sqrt{t}}+1\right).
\end{align}
Thus, combining the estimates of $J_1,J_2$ and $J_3$ with \eqref{Abnonsm}, we have
\begin{eqnarray*}
\delta&\leq& C\bigl[\log(t^{-1})A_1+\log(t^{-1})A_2+A^3A_3\biggl(1+\log(t^{-1})\sum\limits_{i=1}^q\E|W_j|\biggr)\\
&&+\frac{\delta}{\sqrt{t}}A^3A_3+aA^3A_3\bigl(\frac{A}{\sqrt{t}}+1\bigr)\bigr]+a\sqrt{t}.
\end{eqnarray*}
Setting $\sqrt{t}=2CA^3A_3$, provided it is less than 1, simple manipulations yield the result for $\Sigma=\text{Id}$. If $t>1$ for the choice above, then $C$ should be enlarged as necessary.\\
For general $\Sigma$ we can standardize $Y=\Sigma^{1/2}W$. With the conditions on $\G$ we have that $g\bigl(\Sigma^{1/2}x\bigr)\in\G$. 
Hence the bounds \eqref{GoeAb} and \eqref{GoeAb2} can be applied.
The proof now continues as for the $\Sigma=\text{Id}$ case, but with the standardized variables. 
\end{proof}

%In \cite{Eichelsbacher/Loewe:2010}, as well as in \cite{Chatterjee/Shao:2009}, Stein's method of exchangeable pairs was introduced for non-normal distributional approximations. 
Note that we call a function $f$ {\it regular} if $f$ is finite on the interval $I=(a,b)$, $-\infty\leq a<b\leq \infty$, it is defined on and, at any interior point of $I$, $f$ possesses a right-hand limit and a left-hand limit. Further $f$ possesses a right-hand limit $f(a+)$ at the point $a$ and a left-hand limit $f(b-)$ at the point $b$. We will work with the class of functions introduced in \cite{DiaconisStein:2004} for which we let $p$ be a regular, strictly positive density on $I$. We suppose that this density has a derivative $p'$ that is also regular on $I$ with countable many sign changes. Furthermore $p'$ should be continuous at the sign changes and $\int_I p(x) | \log(p(x)) | dx<\infty$. 
Additionally we assume that
\begin{equation} \label{psi}
\psi(x):=\frac{p'(x)}{p(x)}
\end{equation}
is regular. A density $p$ fulfilling these conditions will be called {\it nice}. In \cite{DiaconisStein:2004} it is proven that a random variable $Z$ is distributed according to the density $p$ if and only if $\mathbb{E}\left[f'(Z)+\Psi(Z)f(Z)\right]=f(b-)p(b-)-f(a+)p(a+)$ for a suitably chosen class of functions $f$. The corresponding Stein-identity is
\begin{equation} \label{Steineqp}
f'(x)+\psi(x) f(x)=g(x)-P(g),
\end{equation}
where $g$ is a measurable function for which $\int\limits_I |g(x)| p(x) \, dx < \infty$, $P(x):=\int\limits_{-\infty}^x p(y)\,dy$ and 
$P(g):=\int\limits_I g(y)p(y) \, dy$.
For the proof of Theorem \ref{TT} we will apply Theorem 2.4 and Theorem 2.5 in \cite{Eichelsbacher/Loewe:2010}.

\begin{theorem} \label{EL} Let $(W,W')$ be an exchangeable pair of real-valued random variables such that
\begin{equation} \label{pStein}
\mathbb{E} \left[W'-W\mid W\right]= \lambda\psi(W)-R(W)
\end{equation}
for some random variable $R=R(W)$, $0<\lambda<1$ and $\psi$ as in \eqref{psi} with $p$ being a {\it nice} density. 
Let $p_W$ be a probability distribution such that a random variable $Z_W$ is distributed according to $p_W$ if and only if
\begin{align}\label{pStein2}
\E \bigl( \E[W \psi(W)] f'(Z_W) + \psi(Z_W) f(Z_W) \bigr)=0
\end{align} for a suitably chosen class of functions. 
%Let $Z$ be a random variable distributed according to $p$.\\

{\bf (1)}: Let us assume that for any absolutely continuous function $g$ the solution $f_g$ of \eqref{pStein2} satisfies
$$
\parallel f_g \parallel \leq c_1 \parallel g'\parallel, \quad \parallel f_g'\parallel \leq c_2 \parallel g'\parallel \quad \text{and} \quad
\parallel f_g''\parallel\leq c_3\parallel g'\parallel.
$$
Then for any uniformly Lipschitz function $g$, we obtain $| \mathbb{E} \left[g(W)\right]-\mathbb{E}\left[g(Z_W)\right]| \leq \delta \parallel g'\parallel$
with
\begin{equation} \label{boundp}
\delta: = \frac{c_2}{2 \lambda} \bigl( \V \bigl( \E[(W-W')^2|W] \bigr) \bigr)^{1/2} + \frac{c_3}{4 \lambda} \E |W-W'|^3 + \frac{c_1+c_2 \sqrt{\E(W^2)}}{\lambda}
\sqrt{ \E (R^2)}.
\end{equation}
{\bf (2)}: Let us assume that for any function $g(x) := 1_{\{x \leq z\}}(x)$, $z \in \R$, the solution $f_z$ of \eqref{pStein2} satisfies
$$
|f_z(x)| \leq d_1, \quad |f_z'(x)| \leq d_2 \quad \text{and} \quad
|f_z'(x)-f_z'(y) | \leq d_3
$$
and
\begin{equation} \label{addcond}
|(\psi(x) \, f_z(x))'| = \bigl| ( \frac{p'(x)}{p(x)} \, f_z(x))' \bigr| \leq d_4
\end{equation}
for all real $x$ and $y$, where $d_1, d_2, d_3$ and $d_4$ are constants. Then we obtain for any $A >0$
\begin{eqnarray} \label{kolall}
\sup_{t \in \R} \big| P(W \leq t) - \int_{-\infty}^t p_W(t) \, dt \big|  & \leq & \frac{d_2}{2 \lambda}  \bigl( \V \bigl( \E[(W'-W)^2 |W] \bigr) \bigr)^{1/2} \nonumber \\
& & \hspace{-2cm} + \big( d_1 + d_2 \sqrt{\E(W^2)} + \frac 32 A \bigr)  \frac{\sqrt{\E(R^2)}}{\lambda} + \frac{1}{\lambda} \bigl( \frac{d_4 A^3}{4} \bigr) \\
& & \hspace{-2cm} +  \frac{3A}{2} \E (|\psi(W)|)
 +   \frac{d_3}{2 \lambda} \E \bigl( (W-W')^2 1_{\{|W-W'| \geq A\}} \bigr). \nonumber
\end{eqnarray}
\end{theorem}

\section{Auxiliary results}
Let us fix the convention that for $k,t \in\{1,\ldots,q\}$ we will always write $\sum\limits_{k\neq m}$ instead of 
$\sum\limits_{\stackrel{k=1}{k\neq m}}^q$, $\sum\limits_{k,t\neq m}$ instead of $\sum\limits_{\stackrel{k,t=1}{k\neq m, t \neq m}}^q$ and so on. 
%For the proof of the multivariate normal approximations we will apply Theorem \ref{RR} and Theorem \ref{Kolmogorov}, respectively. In the introduction
%we already presented the construction of the exchangeable pair for $W$, the rescaled empirical spin vector of the Curie-Weiss-Potts model.
%We have already seen in \eqref{motivation} that this construction will lead us to $G_{\beta,h}$, defined in \eqref{Gfunktion}.
%We collect some further results on this function. 
First we state a result on the structure of the minimizers of $G_{\beta,h}$ (see Definition \ref{Gfunktion}), determined
in several papers, collected in \cite{Gandolfo:2010}.

\begin{prop} \label{ChMi}
Let $\beta,h\geq 0$ and let $x$ be a global minimum of $ G_{\beta,h}$.
\begin{enumerate}
\item The vector $x$ has the coordinate $\min(x_i)$ repeated $q-1$ times at least.
\item If $h>0$, then $x_1>x_i$, for all $i\in \{2,\ldots,q\}$.
\item The inequality $\min(x_i)>0$ holds.
\item For any $q\geq 3$ and any $(\beta,h)$, or $q=2$ and $(\beta,h)\neq (\beta_c,0)$, where $\beta_c$ denotes the critical temperature, one has $\min(x_i)<1/\beta$.
\end{enumerate}
\end{prop}

An important identity is the following simple statement.

\begin{lemma} \label{expGD}
For $u\in\mathbb{R}^q$, we obtain
$$
\frac{\exp\left(\beta u_m+h\delta_{m,1}\right)}{\sum\limits_{k=1}^q\exp\left(\beta u_k+h\delta_{k,1}\right)}= u_m-\frac{1}{\beta}\frac{\partial}{\partial u_m}G_{\beta,h}(u).
$$
\end{lemma}

Direct calculation yields
$$
\frac{\partial}{\partial u_m}G_{\beta,h}(u)=\frac{\partial}{\partial u_m}\left(\frac{\beta}{2}\langle u,u\rangle -
\log\left(\sum\limits_{k=1}^q\exp\left(\beta u_k+h\delta_{k,1}\right)\right)\right)
=\beta u_m-\beta \frac{\exp\left(\beta u_m+h\delta_{m,1}\right)}{\sum\limits_{k=1}^q\exp\left(\beta u_k+h\delta_{k,1}\right)}.
$$
Rearranging the equality gives the result.
%\end{proof}

Using the notation $m(\sigma)=(m_1(\sigma),\ldots,m_q(\sigma))$ with
\begin{equation} \label{mit}
m_i(\sigma):=\frac{1}{n}\sum_{j=1}^n \delta_{\sigma_j,i} \quad \text{and} \quad
m_{i,t}(\sigma):=\frac{1}{n}\sum_{j \neq t}^n \delta_{\sigma_j,i}
\end{equation}
we obtain the following lemma.

\begin{lemma} \label{ConD}
For arbitrary $i\in \{1,\ldots,q\}$ we have
\begin{center}
$P_{\beta,h,n}\left(\sigma_j=i\mid (\sigma_t)_{t\neq j}\right)=
\begin{cases} \frac{\exp\left(\beta m_{i,j}(\sigma)\right)}{\sum\limits_{k=1}^q\exp\left(\beta m_{k,j}(\sigma)+h\delta_{k,1}\right)}, \quad i \in \{2,\ldots,q\}; \\ 
\frac{\exp\left(\beta m_{i,j}(\sigma)+h\right)}{\sum\limits_{k=1}^q\exp\left(\beta m_{k,j}(\sigma)+h\delta_{k,1}\right)}, \quad i=1. \end{cases}$
\end{center}
\end{lemma}

\begin{proof}
For $x_1,\ldots,x_n \in \{1,\ldots,q\}$ we have
$$
P_{\beta,h,n}\left(\sigma_j=i\mid (\sigma_t)_{t\neq j}\right) =  
\frac{P_{\beta,h,n}\left(\{\sigma_j=i\}\cap \{(\sigma_t)_{t\neq j}\}\right)}{P_{\beta,h,n}\left(\{(\sigma_t)_{t\neq j}\}\right)}.
$$
For any fixed $x_1, \ldots, x_{j-1}, x_{j+1}, \ldots, x_n$ we obtain
\begin{eqnarray*}
& & \frac{P_{\beta,h,n}\left(\sigma_1=x_1,\ldots,\sigma_j=i,\ldots,\sigma_n=x_n\right)}{P_{\beta,h,n}
\left(\sigma_1=x_1,\ldots,\sigma_{j-1}=x_{j-1},\sigma_{j+1}=x_{j+1},\ldots,\sigma_n=x_n\right)}\\
&=& \frac{Z_{\beta,h,n}^{-1}\exp\left(\frac{\beta}{2n}\sum\limits_{l,t\neq j}^n\delta_{x_l,x_t}+\frac{\beta}{2n}+
\frac{\beta}{n}\sum\limits_{l\neq j}^n\delta_{x_l,i}+h\sum\limits_{l\neq j}^n
\delta_{x_l,1}+h\delta_{i,1}\right)}{\sum\limits_{k=1}^qZ_{\beta,h,n}^{-1}\exp\left(\frac{\beta}{2n}
\sum\limits_{l,t\neq j}^n\delta_{x_l,x_t}+\frac{\beta}{2n}+\frac{\beta}{n}
\sum\limits_{l\neq j}^n\delta_{x_l,k}+h\sum\limits_{l\neq j}^n\delta_{x_l,1}+h\delta_{k,1}\right)}.
\end{eqnarray*}
Canceling equivalent terms in the numerator and denominator and finally distinguishing between $i=1$ and $i\neq 1$ yields the result.
\end{proof}

In the case $h=0$ in \cite[Proposition 2.2]{Ellis/Wang:1990} it is proven that the Hessian $D^2 G_{\beta,0}(x_0)$ of $G_{\beta,0}$
is positive definite if $x_0$ is a global minimum point, and hence invertible. In \cite[Lemma 2]{Wang:1994} it is stated that 
$D^2 G_{\beta,h}(x_0)$ is positive definite for any $\beta > 0$ and $h \geq 0$, if $x_0$ is a global minimum point. 
However this result is not correct. The non-degeneracy of $G_{\beta,h}$ at its minimum points for any $(\beta,h) \not= (\beta_0, h_0)$ is stated
next and will be proven in the Appendix.

\begin{lemma} \label{wangimprovement}
For all $q>2$ let $x_s\in\mathbb{R}^q$ denote a global minimum point of $G_{\beta,h}$. Then %, if $(\beta,h)\neq (\beta_0,h_0)$, 
%$\beta$ never takes one of the values $\{\frac{1}{x_{s,q}},\frac{1}{qx_{s,1}x_{s,q}}\}$ (implying that 
$D^2 G_{\beta,h}(x_s)$ is positive definite for any $(\beta,h)\neq (\beta_0,h_0)$.
\end{lemma}

For the rescaled empirical spin vector of the Curie-Weiss-Potts model, appearing in Theorems \ref{THUM}, \ref{THUM2} and \ref{THMM}, we can bound higher
order moments as follows.

\begin{lemma} \label{EnMo}
For $W=(W_1,\ldots,W_q)$ as in Theorems \ref{THUM}, \ref{THUM2} and for $W^{(i)}$ defined in \eqref{Wi} in Theorem \ref{THMM}  we obtain that for any $l\in \mathbb{N}$ and $j\in\{1,\ldots,q\}$
$$
\mathbb{E} \big| W_j^l \big| \leq \text{const.} (l), \quad \mathbb{E} \big| \bigl( W_j^{(i)} \bigr)^l \big| \leq \text{const.} (l).
$$
\end{lemma}

\begin{proof}
We consider a well known transformation, sometimes called the {\it Hubbard-Stratonovich transformation}, expressing the distribution
of $L_n$ in the Curie-Weiss-Potts model in terms of $G_{\beta,h}$. For $\beta >0$ we pick a random vector $Y$ in a way 
that $\mathcal{L}(Y)$ equals a $q$-dimensional centered Gaussian vector with covariance matrix $\beta^{-1} \text{Id}$ and $Y$ is chosen to be independent from $N$. 
$\text{Id}$ denotes the $q\times q$ identity matrix. According to a simple adaption of Lemma 3.2 in \cite{Ellis/Wang:1990}, for any point $m\in\mathbb{R}^q$ and 
$\gamma\in\mathbb{R}$ and any $n\in\mathbb{N}$ we have
$$
\mathcal{L}\left(\frac{Y}{n^{1/2-\gamma}}+\frac{n(N/n-m)}{n^{1-\gamma}}\right) 
= \exp\left[-n G_{\beta,h}\left(m+\frac{y}{n^{\gamma}}\right)\right]dy
\, \left(\int\limits_{\mathbb{R}^q}\exp\left[-n G_{\beta,h} \left(m+\frac{y}{n^{\gamma}}\right)\right]dy\right)^{-1}.
$$
Lemma 3.2 in \cite{Ellis/Wang:1990} presented this identity only for $h=0$. The calculations for any $h \not= 0$ are omitted.
Applying this result for $\gamma=\frac{1}{2}$ and $m=x_0$ (or any other minimum point of $G_{\beta,h}$) does not change the finiteness of any of the moments of the $W_i$. 
Thus, the new measure has the density
$$
\exp\left[-n G_{\beta,h}\left(x_0+\frac{y}{n^{1/2}}\right)\right]dy\left(\int\limits_{\mathbb{R}^q}\exp\left[-n G_{\beta,h}\left(x_0+\frac{y}{n^{1/2}}\right)\right]dy\right)^{-1}.
$$
Using second order multivariate Taylor expansion of $G_{\beta,h}$ and the fact that 
$x_0$ is a global minimum point of $G_{\beta,h}$ we see that the density of the new measure with respect to Lebesgue measure is given by 
$\text{const.} \exp\left[- \frac 12 \langle y,D^2 G_{\beta,h}(x_0) \, y \rangle \right]$
(up to negligible terms). With Lemma \ref{wangimprovement} we know that for any $(\beta,h) \not= (\beta_0,h_0)$ the Hessian is positive definite, if $x_0$ is a global minimum point. This fact combined with the transformation of integrals yields 
that a measure with this density has moments of any finite order.
\end{proof}

For the random variables $T$ and $V$ in Theorems \ref{TT} and \ref{TV}, we can bound higher order moments as well.

\begin{lemma} \label{EnMo2}
Consider the extremity $(\beta,h)=(\beta_0,h_0)$. For $T$ and $V$ as in \eqref{defTV} we obtain that for any $l\in \mathbb{N}$ and $j\in\{1,\ldots,q\}$
$$
\mathbb{E} \big| V_j^l \big| \leq \text{const.} (l), \quad \mathbb{E} | T^l | \leq \text{const.} (l).
$$
\end{lemma}

\begin{proof}
Note that with $V \in {\mathcal M} \cap u^{\bot}$ we obtain $V_1=0$. Hence $W_1 = (1-q) n^{1/4} T$. Therefore
$$
V = \frac{1}{n^{1/2}} \bigl( N-nx-n^{3/4}Tu \bigr) = W- n^{1/4} Tu = (0, W_2 + \frac{1}{(q-1)} W_1, \ldots, W_q + \frac{1}{(q-1)} W_1).
$$
Since $\bar{W} :=(W_1, \ldots, W_q) \in {\mathcal M}$, we have $W_1 = - \sum_{k=2}^q W_k$. 
We try to check that $\bar{V} := (V_2,\ldots, V_q)$ has finite moments. Thus is suffices to check, that $(W_2, \ldots, W_q)$ has
finite moments. Now we define $G_{q-1}(x)$, $x \in \R^{q-1}$, to be the restriction of $G_{\beta_0,h_0}$ on the last $q-1$ coordinates (the first coordinate will be fixed to $1/2$ in the sequel).
Again we apply the  Hubbard-Stratonovich transformation introduced in the proof of Lemma \ref{EnMo}. We choose a $q-1$ dimensional Gaussian vector $Y$
with covariance matrix $\beta_0^{-1} \text{Id}_{q-1}$ and independent of $\bar{W}$.  With $x=( 1/2(q-1), \ldots, 1/2(q-1)) \in \R^{q-1}$
we have that the law of $Y + \bar{W}$ has the density
$$
\exp\left[-n G_{q-1} \left( x+\frac{y}{n^{1/2}} \right) \right] dy \, \left( \int\limits_{\mathbb{R}^{q-1}}
\exp\left[-n G_{q-1} \left( x+\frac{y}{n^{1/2}} \right) \right] dy \right)^{-1}.
$$
Using second order multivariate Taylor expansion of $G_{q-1}$ and the fact that $(\nabla G_{q-1})(x)=0$ ( $(1/2,x) \in \R^q$ is a global minimum point of $G_{\beta_0,h_0}$),
we see that the density of the new measure with respect to Lebesgue measure is given by 
$\text{const.} \exp\left[- \frac 12 \langle y,D^2 G_{q-1}(x) \, y \rangle \right]$ (up to negligible terms). 
Using the formulas for the second partial derivatives of $G_{\beta_0,h_0}$, see Remark \ref{extrem} in the Appendix,
we obtain that
$$
D^2 G_{q-1}(x) = \frac{4}{q^2} \bigl( \1_{q-1} + \text{Id}_{q-1} (q-1)(q-2) \bigr),
$$ 
where $\1_{q-1}$ denotes the $(q-1) \times (q-1)$ matrix with all entries equal to 1. It is an immediate computation that
$$
\det(D^2 G_{q-1}(x)) = \biggl( \frac{4(q-1)(q-2)}{q^2} \biggr)^{q-2} \biggl(  \frac{4(q-1)(q-2)}{q^2} + \frac{4(q-1)}{q^2} \biggr),
$$
which shows the invertibility of $D^2 G_{q-1}(x)$ for any $q \geq 3$.
Thus $D^2 G_{q-1}(x)$ is positive definite. This fact combined with the transformation of integrals yields 
that a measure with this density has moments of any finite order.

For $(1-q)T= n^{-1/4} W_1$ we apply the Hubbard-Stratonovich transform with $\gamma = 1/4$.
Take a Gaussian random variable with expectation zero and 
variance $\beta_0^{-1}$, independent of $W_1$. The distribution of $n^{-1/4} Y + T$ 
has a density proportional to $\exp( -n G_1 \bigl( c_q 1/2 + y/n^{1/4} \bigr) \bigr)$ with some constant $c_q$
only depending on $q$ and $G_1$ being the restriction of $G_{\beta_0,h_0}$ to the first component. A fourth
order Taylor expansion similar to \eqref{huebsch} will give $G_1(x+t) = G_1(x) + \frac{1}{24} G_1^{(4)} (x + \alpha t) t^4$ for
some $\alpha \in (0,1)$. Hence we conclude that a measure with a Lebesgue-density given by $\text{const.} \exp( -y^4)$ has moments of any finite order.
We omit the details.
\end{proof}

\section{Proofs of the theorems}

%Interesting enough the Curie-Weiss-Potts model will be an example to demonstrate the power of the approach in \cite{ReinertRoellin:2009}.
Constructing an exchangeable pair in the Curie-Weiss-Potts model to obtain an approximate linear regression property \eqref{regressioncond}
leads us to the function $G_{\beta, h}$. 
Let $q >2$, $h=0$ and $\beta < \beta_c$, and let $x_0$ denote the unique global minimum point of $G_{\beta, 0}$, see Theorem \ref{Minima}.

We produce a spin collection $\sigma'=(\sigma'_i)_{i\geq 1}$ via a {\it Gibbs sampling procedure}. Let $I$ be uniformly distributed over $\{1,\ldots,n\}$ and 
independent from all other random variables involved. We will now replace the spin $\sigma_j$ by $\sigma'_j$ drawn from the conditional distribution 
of the i'th coordinate given $(\sigma_t)_{t\neq j}$, independently from $\sigma_j$. We define
$$
Y_j:=(Y_{j,1},\ldots,Y_{j,q})^t := (\delta_{\sigma_j,1},\ldots,\delta_{\sigma_j,q})^t
$$
and consider
\begin{equation} \label{WStr}
W' := W-\frac{Y_I}{\sqrt{n}}+\frac{Y'_I}{\sqrt{n}}.
\end{equation}
Hence it is not hard to see that $(W,W')$ is an exchangeable pair. This construction will also be evident for all the proofs in this section.
Let $\mathcal{F}:=\sigma(\sigma_1,\ldots,\sigma_n)$. We obtain
\begin{align}
\mathbb{E}\left[W'_i-W_i\mid \mathcal{F}\right]
&=\frac{1}{\sqrt{n}}\mathbb{E}\left[Y'_{I,i}-Y_{I,i}\mid \mathcal{F}\right]\nonumber\\
&=\frac{1}{\sqrt{n}}\frac{1}{n}\sum\limits_{j=1}^n\mathbb{E}\left[Y'_{j,i}-Y_{j,i}\mid \mathcal{F}\right]\nonumber\\
&=-\frac{1}{\sqrt{n}}\frac{1}{n}\sum\limits_{j=1}^nY_{j,i}+
\frac{1}{\sqrt{n}}\frac{1}{n}\sum\limits_{j=1}^n\mathbb{E}\left[\delta_{\sigma'_j,i}\mid \mathcal{F}\right]. \nonumber
\end{align}
Using our construction we obtain with Lemma \ref{ConD} 
$$ \E \left[\delta_{\sigma'_j,i}\mid \mathcal{F}\right] =\mathbb{E}\left[\delta_{\sigma_j,i}\mid (\sigma_t)_{t\neq j}\right]
=P_{\beta,0,n}\left(\sigma_j=i\mid (\sigma_t)_{t\neq j}\right) = \frac{\exp\left(\beta m_{i,j}(\sigma)\right)}{\sum\limits_{k=1}^q\exp\left(\beta m_{k,j}(\sigma)\right)},
$$
with $m_{i,j}(\sigma) = \frac 1n \sum_{l \not= j}^n \delta_{\sigma_l,i}$. 
With Lemma \ref{expGD} we obtain for any $i=1, \ldots, q$
\begin{eqnarray} \label{motivation2}
\E [ W_i' - W_i | \mathcal{F} ] & = &  - \frac 1n W_i - \frac{x_{0,i}}{\sqrt{n}} + R_n^{(1)}(i) + \frac{1}{\sqrt{n}} \biggl( m_i(\sigma) - 
\frac{1}{\beta} \frac{\partial}{\partial u_i} G_{\beta,0} (m(\sigma)) \biggr) \nonumber \\ 
& = & -\frac{1}{\sqrt{n}} \frac{1}{\beta} \frac{\partial}{\partial u_i} G_{\beta,0} (m(\sigma)) +  R_n^{(1)}(i)
\end{eqnarray}
with 
\begin{equation} \label{R1i}
R_n^{(1)}(i) := \frac{1}{\sqrt{n}}\frac{1}{n}\sum\limits_{j=1}^n \left[\frac{\exp\left(\beta m_{i,j}(\sigma)\right)}{\sum\limits_{k=1}^q
\exp\left(\beta m_{k,j}(\sigma)\right)}-\frac{\exp\left(\beta m_{i}(\sigma)\right)}{\sum\limits_{k=1}^q\exp\left(\beta m_{k}(\sigma)\right)}\right],
\end{equation}
where $m_i(\sigma)$ and $m_{i,j}(\sigma)$ are defined as in \eqref{mit}.
We have used 
\begin{equation} \label{easy}
m_i(\sigma) - x_{0,i} = \frac{W_i}{\sqrt{n}}.
\end{equation}

\begin{proof}[Proof of Theorem \ref{THUM}]
Our goal is to apply Theorem \ref{RR}. 
%First,  given $W$, we construct a coupling $W'$ and will calculate $\Lambda$ and $R$ to get
%the approximative regression identity \eqref{regressioncond}. 
We will first of all deal with the case $h =0$. Hence by Theorem \ref{Minima} we
have $\beta < \beta_c$ and $x_0=( 1/q, \ldots, 1/q)$ being the unique minimum point of $G_{\beta,0}$.
We apply \eqref{taylorsecond} (see Appendix) to the first summand in \eqref{motivation2}.  Since $x_0$ is a global minimum of 
$ G_{\beta,0}$ we have $ \bigl( \frac{\partial}{\partial u_i} G_{\beta,0} \bigr)(x_0)=0$.
Hence the first summand in \eqref{motivation2} is equal to 
$$
- \frac{1}{\beta \, n} \left( \frac{\partial^2}{\partial^2 u_i} G_{\beta,0} \right) (x_0) \, W_i 
- \frac{1}{\beta \, n} \sum \limits_{k\neq i} \left( \frac{\partial^2}{\partial u_i\partial u_k} G_{\beta,0} \right) (x_0) \, W_k + R_n^{(2)}(i)
$$
with 
\begin{equation} \label{R2i}
R_n^{(2)}(i) : =\mathcal{O}\left(\frac{1}{\sqrt{n}}\left(\frac{W_i}{\sqrt{n}}\right)^2\right)-\sum\limits_{k\neq i}\mathcal{O}
\left(\frac{1}{\sqrt{n}}\frac{W_i}{\sqrt{n}}\frac{W_k}{\sqrt{n}}\right)-\sum\limits_{k,t\neq i}\mathcal{O}
\left(\frac{1}{\sqrt{n}}\frac{W_k}{\sqrt{n}}\frac{W_t}{\sqrt{n}}\right).
\end{equation}
Summarizing with $R(i) := R_n^{(1)}(i) + R_n^{(2)}(i)$ we have
\begin{eqnarray*} %\label{hzero}
\mathbb{E}\left[W'_i-W_i\mid \mathcal{F}\right] & = &
- \frac{1}{\beta \, n} \left( \frac{\partial^2}{\partial^2 u_i} G_{\beta,0} \right) (x_0) \, W_i 
- \frac{1}{\beta \, n} \sum \limits_{k\neq i} \left( \frac{\partial^2}{\partial u_i\partial u_k} G_{\beta,0} \right) (x_0) \, W_k + R(i) \nonumber \\
& = & - \frac{1}{\beta \, n} \langle \bigl[ D^2 G_{\beta,0} (x_0) \bigr]_i , W \rangle + R(i), %R_n^{(1)}(i) + R_n^{(2)}(i)
\end{eqnarray*}
where $\langle \cdot, \cdot \rangle$ denotes the Euclidean scalar-product and $\bigl[D^2 G_{\beta,0}(x_0)\bigr]_i$ the $i$-th row of the matrix $D^2 G_{\beta,0}(x_0)$.
We obtain
\begin{equation} \label{hzero}
\mathbb{E}\left[W'-W\mid \mathcal{F}\right] = - \frac{1}{\beta \, n} \bigl[ D^2 G_{\beta,0} (x_0) \bigr] W + R(W)
\end{equation}
with  $R(W) = (R(1), \ldots, R(q))$.
We define $\Lambda= \frac{1}{\beta \, n} \bigl[ D^2 G_{\beta,0} (x_0) \bigr]$. %and $R(W) = (R(1), \ldots, R(q))$.
With \cite[Proposition 2.2]{Ellis/Wang:1990}, $D^2 G_{\beta,0}(\nu)$ is positive definite for any $\beta >0$ and any global minimum point
$\nu$ and therefore
$\Lambda$ is invertible (alternatively one easily sees that $\Lambda$ is a matrix of the form given in Lemma \ref{Matrix} and the determinant is
$\frac{1}{n^q} (1- \beta/q)^{q-2} (1 - \beta/q)$ which is non-zero because $\beta_c <q$ with $\beta_c$ given in \eqref{criticaltemp}, and therefore $\beta \not= q$,
see Lemma \ref{hessiancal}). 
Hence \eqref{regressioncond} is fulfilled and we are able to apply Theorem \ref{RR}. In order to calculate the bound given there 
we need to estimate $\lambda^{(i)}$ as well as the order of the terms 
$A,B$ and $C$. Note that often in an application of Theorem \ref{RR} it might be tedious to calculate $\Lambda$ (and $\Sigma$) and it is
not clear whether the calculations have been carried out correctly. In Remark \ref{heuristic}, we will point out, that there is a nice
heuristic in the Curie-Weiss-Potts model expecting $\Lambda$ as it comes out.

Obviously we have $\lambda^{(i)}=\mathcal{O}(n)$. We continue by estimating $C$ in Theorem \ref{RR}. First we consider $R_n^{(1)}(i)$ defined in \eqref{R1i}.
\begin{eqnarray*}
| R_n^{(1)}(i) | &\leq& \frac{1}{\sqrt{n}}\frac{1}{n}\sum\limits_{j=1}^n \bigg| 
\frac{\exp\left(\beta m_{i,j}(\sigma)\right)}{\sum\limits_{k=1}^q\exp\left(\beta m_{k,j}(\sigma)\right)}-
\frac{\exp\left(\beta m_{i}(\sigma)\right)}{\sum\limits_{k=1}^q\exp\left(\beta m_{k}(\sigma)\right)} \bigg| \\
&\leq& \frac{1}{\sqrt{n}}\frac{1}{n}\sum\limits_{j=1}^n\sum\limits_{k\neq i} \big| \exp\left(\beta m_{i,j}(\sigma)+\beta m_{k}(\sigma)\right)-
\exp\left(\beta m_{i}(\sigma)+\beta m_{k,j}(\sigma)\right) \big|. 
\end{eqnarray*}
Using the inequality
\begin{center}
$ \big| \exp(\alpha x)-\exp(\alpha y) \big| \leq \frac{|\alpha|}{2}(\exp(\alpha x)+\exp(\alpha y))\, |x-y|$  , for all $\alpha,x,y\in\mathbb{R}$,
\end{center}
we obtain
$$
| R_n^{(1)}(i) | \leq \frac{1}{\sqrt{n}}\frac{1}{n}\beta e^{2\beta}\sum\limits_{j=1}^n\sum\limits_{k\neq i}
\big| m_{i,j}(\sigma)+ m_{k}(\sigma)- m_{i}(\sigma) - m_{k,j}(\sigma) \big|. 
$$
Consider the first summand $j=1$. In case $\sigma_1=i$, we have for all $k\neq i$ that $m_{k,1}(\sigma)=m_{k}(\sigma)$, and therefore 
\begin{center}
$\sum\limits_{k\neq i} | m_{i,1}(\sigma)+ m_{k}(\sigma)- m_{i}(\sigma)- m_{k,1}(\sigma) | = (q-1) \frac{\delta_{\sigma_1,i}}{n}.$
\end{center}
If $\sigma_1\neq i$, then there is a $t\neq i$ with $m_{t,1}(\sigma)\neq m_t(\sigma)$ and for all $k\neq t$: $m_{k,1}(\sigma)=m_k(\sigma)$. 
By similar observation we have
\begin{center}
$\sum\limits_{k\neq i} |m_{i,1}(\sigma)+ m_{k}(\sigma)- m_{i}(\sigma)- m_{k,1}(\sigma) | \leq (q-1) \frac{\delta_{\sigma_1,t}}{n}.$
\end{center}
The same observation can be made for any other $j\in\{1,\ldots,n\}$. With $ | \delta_{\sigma_j,t} | \leq 1$ we get
$$
| R_n^{(1)}(i)| \leq \frac{1}{\sqrt{n}} \frac{1}{n} (q-1) \beta e^{2\beta} =\mathcal{O}(n^{-3/2}).
$$
Since $W \in {\mathcal M}$, see \eqref{hyperM}, we get $\sum_{k \not= i} W_k = - W_i$.
By Lemma \ref{EnMo} we know that $ \E |W_i^2| \leq \text{const.} (2)$ and therefore we obtain that $\mathbb{E} |R_n^{(2)}(i)|$ in \eqref{R2i} is $\mathcal{O}(n^{-3/2})$.
Thus the Cauchy-Schwartz inequality yields $\mathbb{E}[R(i)^2]=\mathcal{O}(n^{-3})$ for all $i\in \{1, \ldots, q\}$. We have
\begin{center}
$C=\sum\limits_{i=1}^q\lambda^{(i)}\sqrt{\mathbb{E}\left[R_i^2\right]}=\mathcal{O}(n^{-1/2}).$
\end{center}
The next thing we notice is that $|W_i'-W_i| =\frac{1}{\sqrt{n}} \big| Y_{I,i}'-Y_{I,i} \big| \leq \frac{1}{\sqrt{n}}$ for all $i$. 
Thus we easily obtain the bound $B=\mathcal{O}(n^{-1/2})$. It remains to calculate and to estimate the conditional variance in $A$. 
This is a bit more involved. We have
\begin{eqnarray} \label{ais}
\mathbb{E}[(W_i'-W_i)(W_j'-W_j) \mid \mathcal{F}]&=& \frac{1}{n^3}\sum\limits_{t,k=1}^nY_{k,i}Y_{t,j}+\frac{1}{n^3}
\sum\limits_{t,k=1}^n\mathbb{E}[Y'_{k,i}Y'_{t,j}\mid \mathcal{F}]\nonumber\\
& -& \frac{2}{n^3}\sum\limits_{t,k=1}^nY_{k,i}\mathbb{E}[Y'_{t,j}\mid \mathcal{F}]=:A_1+A_2+A_3. \nonumber
%-\frac{1}{n^3}\sum\limits_{t,k=1}^nY_{t,j}\mathbb{E}[Y'_{k,i}\mid \mathcal{F}]\nonumber\\
\end{eqnarray}
Hence we have to bound the variances of these terms. By definition $\mathbb{V}[A_1]=\frac{1}{n^2}\mathbb{V}[m_i(\sigma)\, m_j(\sigma)]$.
Now
\begin{eqnarray*}
\mathbb{V}[m_i(\sigma)\, m_j(\sigma)] & =& \mathbb{V} \bigl( \frac{W_i \, W_j}{n} + \frac{W_i}{\sqrt{n}} x_{0,j} + \frac{W_j}{\sqrt{n}} x_{0,i} \bigr) \\
& \leq & \text{const.} \max \bigl( \frac{1}{n^2} \V(W_i \, W_j), \frac 1n \V(W_i) \bigr) \leq \frac{\text{const.}}{n^2} \bigl( \E [ W_i^2 W_j^2] + n \E [W_i^2] \bigr).
\end{eqnarray*}
We make use of Lemma \ref{EnMo} to obtain $\mathbb{V}[m_i(\sigma)\, m_j(\sigma)] = {\mathcal O}(1/n)$ and hence $\V[A_1]= {\mathcal O}(n^{-3})$.
Using a conditional version of Jensen's inequality we have
$$
\mathbb{V}[A_2] \leq \mathbb{E} \bigl( \mathbb{V} \bigl[ \frac{1}{n^3} \sum\limits_{t,k=1}^n Y'_{k,i} Y'_{t,j} \bigr] \mid \mathcal{F} \bigr)
= \mathbb{E} \bigl( \mathbb{V} \bigl[ \frac{1}{n^3}\sum\limits_{t,k=1}^n Y_{k,i}Y_{t,j} \bigr] \mid \mathcal{F} \bigr) 
= \mathbb{V} \bigl( \frac{1}{n^3}\sum\limits_{t,k=1}^nY_{k,i}Y_{t,j} \bigr).
$$
Hence $\V[A_2] = \mathcal{O}(n^{-3})$.
With Lemma \ref{ConD} we get
\begin{eqnarray*}
- A_3/2  & = & \frac{1}{n^3} \sum\limits_{t,k=1}^n Y_{k,i}\, \mathbb{E}[Y'_{t,j}\mid \mathcal{F}] =\frac{1}{n^3}
\sum\limits_{t,k=1}^nY_{k,i} \frac{\exp\left(\beta m_{j,t}(\sigma)\right)}{\sum\limits_{l=1}^q\exp\left(\beta m_{l,t}(\sigma)\right)} \\
&=& \frac{1}{n^3} \sum \limits_{t,k=1}^nY_{k,i} \biggl( \frac{\exp\left(\beta m_{j,t}(\sigma)\right)}{\sum\limits_{l=1}^q\exp\left(\beta m_{l,t}(\sigma)\right)}
-\frac{\exp\left(\beta m_{j}(\sigma)\right)}{\sum\limits_{l=1}^q\exp\left(\beta m_{l}(\sigma)\right)} \biggr) 
+\frac{1}{n^2}\sum\limits_{k=1}^n Y_{k,i}\frac{\exp\left(\beta m_{j}(\sigma)\right)}{\sum\limits_{l=1}^q\exp\left(\beta m_{l}(\sigma)\right)}\\
& = &:M_1+M_2.
\end{eqnarray*}
By using the same estimations as for $R_n^{(1)}(i)$ we obtain
$$
M_1 \leq \frac{1}{n^3} \sum \limits_{t,k=1}^n Y_{k,i} \, (q-1) \beta e^{2\beta} = \frac{1}{n} \, (q-1) \beta e^{2 \beta} \bigl(\frac{W_i}{\sqrt{n}}+x_{0,i}\bigr).
$$
Hence $\V(M_1)= {\mathcal O}(n^{-3})$ by Lemma \ref{EnMo}. We obtain
\begin{eqnarray*}
M_2 & = & \frac 1n m_i(\sigma) \bigl( m_j(\sigma) - \frac{1}{\beta} \frac{\partial}{\partial u_j} G_{\beta,0} (m(\sigma)) \bigr) \\
& = & \frac 1n m_i(\sigma) \, m_j(\sigma) - \frac{1}{\beta n} m_i(\sigma) \biggl( \bigl( \frac{\partial^2}{\partial^2 u_j} G_{\beta,0}\bigr)(x_0) (m_j(\sigma) - x_{0,j})\\
& + & \sum_{k \not= j} \bigl( \frac{\partial^2}{\partial u_j u_k} G_{\beta,0} \bigr)(x_0) (m_k(\sigma) - x_{0,k}) + \sqrt{n} R_n^{(2)}(j) \biggr),
\end{eqnarray*}
where the first equality follows from Lemma \ref{expGD}, the second from \eqref{taylorsecond} and the definition of $R_n^{(2)}(j)$ in \eqref{R2i}.
Hence
$$
M_2 = {\mathcal O} \biggl( \frac 1n m_i(\sigma) \, m_j(\sigma) \biggr) +  {\mathcal O} \biggl( \frac 1n m_i(\sigma) \frac{W_j}{\sqrt{n}} \biggr) + {\mathcal O} \bigl(n^{-1/2}
\, R_n^{(2)}(j) \bigr).
$$
The first two summands are of order ${\mathcal O} \bigl( W_j / n^{3/2} \bigr)$ and the last term is of order ${\mathcal O}(n^{-2})$. Applying Lemma \ref{EnMo}, it
follows that the maximal variance of all the sums in the representation of $M_2$ is of order ${\mathcal O}(n^{-3})$ and therefore $\mathbb{V}(A_3)=\mathcal{O}(n^{-3})$. 
Thus the variance in $A$ of Theorem \ref{RR} can be bounded by 9 times the maximum of the variances of $A_1, A_2, A_3$, which is a constant times 
$n^{-3}$. Thus we obtain
\begin{center}
$A=\sum\limits_{i,j=1}^q\lambda^{(i)}\sqrt{\mathbb{V}\left[\mathbb{E}[(W'_i-W_i)(W'_j-W_j)\mid W]\right]}=\mathcal{O}(n^{-1/2})$.
\end{center}
This completes the proof for $h=0$. Note that we have used the fact that the fourth moment of $W_i$ is bounded. We did not need the finiteness of any higher
moment. We have proved a {\it fourth-moment} Theorem together with a rate of convergence of order $\mathcal{O}(n^{-1/2})$.

If $h \not=0$ we will slightly change the proof. Here are the details. By Theorem \ref{Minima} we know that for $h >0$ and $(\beta,h) \notin h_T$, the
function $G_{\beta,h}$ has a unique global minimum point. Let $x_0$ be the unique global minimum point. Analogously to the first part of our proof we obtain
\begin{eqnarray} \label{hnonzero}
\mathbb{E}\left[W'_i-W_i\mid \mathcal{F}\right] & = &
- \frac{1}{\beta \, n} \left( \frac{\partial^2}{\partial^2 u_i} G_{\beta,h} \right) (x_0) \, W_i 
- \frac{1}{\beta \, n} \sum \limits_{k\neq i} \left( \frac{\partial^2}{\partial u_i\partial u_k} G_{\beta,h} \right) (x_0) \, W_k + R(i,h) \nonumber \\
& = & - \frac{1}{\beta \, n} \langle \bigl[ D^2 G_{\beta,h} (x_0) \bigr]_i , W \rangle + R(i,h) %R_n^{(1)}(i) + R_n^{(2)}(i)
\end{eqnarray}
with $R(i,h) := R_n^{(1)}(i,h) + R_n^{(2)}(i)$ with the new
$$
R_n^{(1)}(i,h) :=
\frac{1}{\sqrt{n}}\frac{1}{n}\sum\limits_{j=1}^n\left[\frac{\exp\left(\beta m_{i,j}(\sigma)+h\delta_{i,1}\right)}{\sum\limits_{k=1}^q\exp\left(\beta m_{k,j}(\sigma)+h\delta_{k,1}\right)}-\frac{\exp\left(\beta m_{i}(\sigma)+h\delta_{i,1}\right)}{\sum\limits_{k=1}^q\exp\left(\beta m_{k}(\sigma)+h\delta_{k,1}\right)}\right]
$$
and the same $R_n^{(2)}(i)$, given in \eqref{R2i}. Again $\Lambda = \frac{1}{\beta n}  \bigl[ D^2 G_{\beta,h}(x_0) \bigr]$. This matrix has a simple structure. 
With Lemma \ref{hessiancal} we obtain 
$$
a =  \frac{1}{\beta \, n} \left( \frac{\partial^2}{\partial^2 u_1} G_{\beta,h} \right) (x_0) = \frac{1-\beta(q-1)x_{0,1}x_{0,q}}{n}, \, \,
b = \frac{1}{\beta \, n} \left( \frac{\partial^2}{\partial u_1\partial u_q} G_{\beta,h} \right) (x_0) =\frac{\beta x_{0,1}x_{0,q}}{n}.
$$
Moreover
$$
d=  \frac{1}{\beta \, n} \left( \frac{\partial^2}{\partial^2 u_q} G_{\beta,h} \right) (x_0) = \frac{1-\beta(x_{0,1}x_{0,q}+(q-2)x_{0,q}^2)}{n}, \, \,
c=\frac{1}{\beta \, n} \left( \frac{\partial^2}{\partial u_2 u_q}G_{\beta,h}\right) (x_0)=\frac{\beta x_{0,q}^2}{n}.
$$
Hence $\Lambda$ has the form \eqref{Matrix} and 
%there:=\begin{pmatrix} a & b & ... & ... & b \\ b & d & c & ...& c \\ ... &  & & ... & ...\\ b & c& ...& c & d  \end{pmatrix}. \label{koLa}
%$$
according to Lemma \ref{DetLe} we have
$$
\det(\Lambda) =
%&= \frac{1}{n^{q-2}}\left[1-\beta(x_{0,1}x_{0,q}-(q-2)x_{0,q}^2)-\beta x_{0,q}^2\right]^{q-2}\cdot\left(a(d+(q-2)c)-(q-1)b^2\right)\nonumber\\
%&= \frac{1}{n^q}\left(1-\beta x_{0,q}\right)^{q-2}((1-(q-1)\beta x_{0,1}x_{0,q})(1-\beta x_{0,1}x_{0,q})-(q-1)\beta^2 x_{0,1}^2 x_{0,q}^2)\nonumber\\
\frac{1}{n^q}\left(1-\beta x_{0,q}\right)^{q-2}(1-q\beta x_{0,1}x_{0,q}).
$$
So if $\beta \notin \{\frac{1}{x_{0,q}},\frac{1}{qx_{0,1}x_{0,q}}\}$, the matrix $\Lambda$ is invertible.  
With Lemma \ref{wangimprovement} we get that $\Lambda$ is invertible for all 
$(\beta,h) \neq (\beta_0,h_0)$ and hence we are able to apply Theorem \ref{RR}.  
The bound of $R_n^{(1)}(i,h)$ is $e^h$ times the bound of $R_n^{(1)}(i)$ implying the same order of $C$. The proof of bounding $B$ is unchanged. Bounding
$A$ needs once more the bound $R_n^{(1)}(i,h)$ and hence the proof is almost the same as in the case $h=0$. 
%For the $q$-dependence we note that $\mathbb{E}\bigl[|W'_i-W_i||W'_j-W_j||W'_k-W_k|\bigr]$ is independent of q and $\lambda^{(i)}=\mathcal{O}(q)$. Thus $B=\mathcal{O}\bigl(q^4\bigr)$ by summing three times to q. Previous estimations show that $R_n^{(1)}=\mathcal{O}(q)$ and because we have two sums from 1 to q for $W_k$ and $W_t$ we have that $R_n^{(2)}=\mathcal{O}\bigl(q^2\bigr)$. Thus $\mathbb{E}\bigl[R(i)^2\bigr]=\mathcal{O}\bigl(q^4\bigr)$. Hence $C=\mathcal{O}\bigl(q^4\bigr)$ by summing  over q. Next we consider $A_1,A_2$ and $A_3$ taken from the proof. While $A_1$ and $A_2$ are independent of q, $A_3$ depends on q via $R_n^{(1)}$ and $R_n^{(2)}$. Summing two times over q, we obtain $A=\mathcal{O}\bigl(q^5\bigr)$. Since the brackets in the bound of Reinert and R\"ollin still contain the parameter q and since $\|\Sigma\|^{1/2}=\mathcal{O}(q)$, the constant C of our theorem satisfies $C=\mathcal{O}\bigl(q^6\bigr)$.
\end{proof}
\begin{proof}[Proof of Theorem \ref{THUM2}]
Since the first part of the proof follows the lines of the proof of Theorem \ref{THUM} we notice that Theorem \ref{Kolmogorov} can be applied. Thus it remains to estimate the bound given there. For the first expression in the bound we notice that $A_1$ is the same expression as the A-term we just calculated for the proof of Theorem \ref{THUM}. Hence, $\log(n)A_1=\mathcal{O}\bigl(\log(n)n^{-1/2}\bigr)$. With Lemma \ref{EnMo} and the estimation for the C-term in Theorem \ref{THUM} we obtain that the second expression is $\mathcal{O}\bigl(\log(n)n^{-1/2}\bigr)$. For the third expression we notice that $a>1$ is a constant and that $A_3=\mathcal{O}(n)$. Using again Lemma \ref{EnMo} combined with the fact that $A=\frac{1}{\sqrt{n}}$ yields that the third expression is also $\mathcal{O}\bigl(\log(n)n^{-1/2}\bigr)$. Likewise we obtain that the fourth expression is $\mathcal{O}\bigl(n^{-1/2}\bigr)$. Combining these estimations yields the result.
\end{proof}

\begin{remark}[Heuristics] \label{heuristic}
By definition of $G_{\beta,h}$, \eqref{Gfunktion}, the Hessian of $G_{\beta,h}$ fulfills $D^2 G_{\beta,h}(x) = \beta \text{Id} - \beta^2 D^2 \Phi (x)$, where
$\Phi$ is the $\log$-moment generating function of the single-spin distribution in the Curie-Weiss-Potts model and $x$ is any minimum point. Hence $D^2 \Phi$ is the
covariance structure of the single-spins, which is
\begin{equation} \label{heusingle}
D^2 \Phi(x) = - \frac{1}{\beta^2} \bigl(D^2 G_{\beta,h}(x) - \beta \text{Id} \bigr).
\end{equation}
We know from Stein's method that
if $(W,W')$ is exchangeable and \eqref{regressioncond} is satisfied with $R=0$ we have
$$
\frac 12 \E[(W'-W)(W'-W)^t] = \Sigma \, \Lambda^t.
$$
On the one hand in the Curie-Weiss-Potts model we have $\Sigma = [D^2 G_{\beta,h}(x)]^{-1} - \beta^{-1} \text{Id}$. 
On the other hand the left hand side describes the empirical covariance structure of the single-spins,
$$
\frac 12 \E[(W_i'-W_i)(W_j'-W_j)] = \frac{1}{2n} \E \bigl(Y_{I,i}'-Y_{I,i} \bigr) \bigl(Y_{I,j}' - Y_{I,j} \bigr).
$$
Therefore with \eqref{heusingle}, heuristically
$$
\frac 12 \E[(W'-W)(W'-W)^t] \approx \frac 1n \frac{1}{\beta^2} \bigl(- D^2 G_{\beta,h}(x) + \beta \text{Id} \bigr) = 
\bigl( [D^2 G_{\beta,h}(x)]^{-1} - \beta^{-1} \text{Id} \bigr) \, \Lambda^t.
$$
If we now choose $\Lambda = \Lambda^t = \frac{1}{\beta \, n} D^2G_{\beta,h}(x)$, the
right hand identity is fulfilled.
\end{remark}

\begin{proof}[Proof of Theorem \ref{THMM}]
The proof uses the fact that the conditional joint distribution of the $(\sigma_i)_i$, 
conditioned on the event $\big\{ \frac{N}{n}\in B(x_i,\epsilon) \bigr\}$, is given by
$$P_{ \beta,h,n,\epsilon}( \sigma)=\frac{1}{Z_{ \beta,h,n,\epsilon}}\exp\left(\frac{ \beta}{2n}\sum\limits_{1\leq i\leq j\leq n}\delta_{ \sigma_i,\sigma_j}+
h\sum\limits_{i=1}^n\delta_{ \sigma_i,1}\right)\mathbf{1}_{B(x_i,\epsilon)}(N/n),
$$
where $Z_{ \beta,h,n,\epsilon}$ denotes a normalization. Thus we are able to start with any minimum point $x_0$ and follow the lines of the proof of 
Theorem \ref{THUM}. 
\end{proof}

\begin{proof}[Proof of Theorem \ref{TT}]
We will apply Theorem \ref{EL}. Obviously the density $p$ is nice. Note that the logarithmic derivative is
$\psi(t) = \frac{p'(t)}{p(t)} = - \frac{16(q-1)^4}{3} \, t^3$.
The solutions $f_g$ of the corresponding Stein equations \eqref{Steineqp} - with respect to absolutely continuous test functions $g$ and 
with respect to $g(x)=1_{\{x \leq z\}}(x)$, $z \in \R$, respectively - fulfill all boundedness assumption of Theorem \ref{EL}.
This was proven in \cite[Lemma 2.2]{Eichelsbacher/Loewe:2010}. By definition of $T$, see \eqref{defTV}, we have
$$
T = \frac{1}{(1-q) n^{3/4}} \bigl( N_1 - n x_1 - \sqrt{n} V_1 \bigl) = \frac{1}{(1-q) n^{3/4}} \biggl( \sum_{i=1}^n Y_{i,1} - n x_1 - \sqrt{n} V_1 \biggr).
$$
We make use of the choice $V \in \mathcal{M}\cap u^{\bot}$. With $V \in \mathcal{M}$ we have $\sum_{i=1}^q V_i=0$ and 
with $\langle V, u \rangle = V_1 (1-q) + \sum_{i=2}^q V_i=0$ we obtain
$$
\sum\limits_{i=2}^q V_i = -V_1 =0.
$$
Constructing an exchangeable pair $(T,T')$ is just the same as in the introduction. $T'$ is a random variable being the same as $T$ except that
we pick an index $I$ uniformly and exchange $Y_{I,1}$ with $Y_{I,1}'$ (for $I=i$ distributed according to the conditional distribution of $Y_{i,1}$ given
$(Y_{j,1})_{j \not= i}$, independently of $Y_{i,1}$).  Now we calculate $\E [T'-T | {\mathcal F}]$ with $\mathcal{F}=\sigma(\sigma_1,\ldots,\sigma_q)$.
\begin{eqnarray*}
\mathbb{E}[T'-T \mid \mathcal{F}] &=& \frac{1}{(1-q) \, n^{7/4}} \sum\limits_{i=1}^n \mathbb{E}\left[ Y_{i,1}'-Y_{i,1} \mid \mathcal{F}\right] \\
& = & \frac{1}{(1-q) \, n^{7/4}} \sum_{i=1}^n \E[ Y_{i,1}' | {\mathcal F}] - \frac 1n T - \frac{x_1}{(1-q) n^{3/4}}. %- \frac{V_1}{(1-q) n^{5/4}}.
\end{eqnarray*}
With Lemma \ref{expGD} we obtain
$$
\frac{1}{(1-q) \, n^{7/4}} \sum_{i=1}^n \E[ Y_{i,1}' | {\mathcal F}] = \frac{1}{(1-q) n^{3/4}} \bigl( m_1(\sigma) - \frac{1}{\beta_0} \frac{\partial}{\partial x_1}
G_{\beta_0, h_0} (m(\sigma)) \bigr) + \frac{1}{(1-q) n^{1/4}} R_n^{(1)}(i,h_0).
$$
Hence using $m_1(\sigma) = x_1 + \frac{(1-q) T}{n^{1/4}}$ and defining $\widetilde{R} :=\frac{1}{(q-1)n^{1/4}}R_n^{(1)}(i,h_0)$
we have
\begin{equation*}
\mathbb{E}[T'-T\mid \mathcal{F}] = - \frac{1}{(1-q)\beta_0 n^{3/4}} \frac{\partial}{\partial x_1 }G_{\beta_0,h_0}\left(x+\frac{Tu}{n^{1/4}}+\frac{V}{n^{1/2}}\right) -\widetilde R.
\end{equation*}
A quite tedious {\it fourth-order} Taylor expansion of $G_{\beta_0,h_0}$ at $x+\frac{Tu}{n^{1/4}}+\frac{V}{n^{1/2}}$ is 
affiliated in the Appendix, see \eqref{finaleins}, which leads to 
\begin{eqnarray} \label{FIN1}
\frac{\partial}{\partial x_1 }G_{\beta_0,h_0}\left(x+\frac{Tu}{n^{1/4}}+\frac{V}{n^{1/2}}\right) &=&
 -\frac{16(q-1)^4}{3 q n^{3/4}} T^3  + \mathcal{O} \left(\sum_{j=2}^q
\frac{V_j^2}{n}\right)\\
&~& + \mathcal{O} \left(\frac{T^4}{n}\right) + 
{\mathcal O} \biggl( f \bigl( V/\sqrt{n}, T/n^{1/4} \bigr) \biggr) \nonumber
\end{eqnarray}
with $f(v,t)$ given in \eqref{arbeit2}.
Hence we obtain $\mathbb{E}[T'-T\mid \mathcal{F}] = \lambda \psi(T) - R$ with
$$
\lambda: = \frac{1}{q(q-1)\beta_0 n^{3/2}}
$$
and
$$
-R:= \mathcal{O} \left(\sum_{j=2}^q
\frac{V_j^2}{n^{7/4}}\right) +\mathcal{O}\left(\frac{T^4}{n^{7/4}}\right) + 
{\mathcal O} \biggl( \frac{1}{n^{3/4}} f \bigl( V/\sqrt{n}, T/n^{1/4} \bigr) \biggr) -\widetilde R.
%\mathcal{O}\left(\frac{V_i^2}{n^{7/4}}\right)+\mathcal{O}\left(\frac{V_i^3}{n^{9/4}}\right)+\mathcal{O}\left(\frac{T^4}{n^{7/4}}\right)-\widetilde R.
$$
Now $0<\lambda<1$ for all $n \in \N$ and thus we can apply Theorem \ref{EL}. The moments of $T$ and $V$ are finite, see Lemma \ref{EnMo2}. 
With $V_1=0$ we get $T =\frac{1}{(1-q) n^{1/4}} \, W_1$. 
Now we are able to compute the expressions of the bound in Theorem \ref{EL}.  
We have
$$
\E [ (T'-T)^2 | T] = \frac{1}{(1-q)^2 n^{1/2}} \E [ (W_1'-W_1)^2 | W].
$$
Reproducing the proof of Theorem \ref{THUM} we get $\mathbb{V}\left(\mathbb{E}[(W_1'-W_1)^2\mid \mathcal{F}]\right)=\mathcal{O}(n^{-5/2})$, using
$\E|W^l| \leq n^{l/4} \E|T^l| = {\mathcal O}(n^{l/4})$, $l \in \N$.
Thus
\begin{align*}
\frac{c_2}{2\lambda}\left(\mathbb{V}\left(\mathbb{E}[(T'-T)^2\mid T]\right)\right)^{1/2}=\mathcal{O}(n^{-1/4}).
\end{align*}
Moreover $\E|T'-T|^3 =\frac{1}{(1-q)^3 n^{3/4}} \frac{1}{n^{3/2}} \E|Y_I'-Y_I| = \mathcal{O}(n^{-9/4})$ and therefore
$\frac{c_3}{4 \lambda} \E|T'-T|^3 = \mathcal{O}(n^{-3/4})$. From the proof of Theorem \ref{THUM} we know that $|R_n^{(1)}(i,h_0)|=  \mathcal{O}(n^{-3/2})$, so
$\widetilde{R} = \mathcal{O}(n^{-7/4})$. Note that by \eqref{arbeit2} we see that the expectation
of ${\mathcal O} \biggl( \frac{1}{n^{3/4}} f \bigl( V/\sqrt{n}, T/n^{1/4} \bigr) \biggr)$ is of order ${\mathcal O} (n^{-2})$.
Summarizing we obtain
$\sqrt{\mathbb{E}[R^2]}=\mathcal{O}(n^{-7/4})$, hence
\begin{align*}
\frac{c_1+c_2\sqrt{\mathbb{E}[T^2]}}{\lambda}\sqrt{\mathbb{E}[R^2]}=\mathcal{O}(n^{-1/4}).
\end{align*}
Hence the $\delta$ in \eqref{boundp} is of order $\mathcal{O}(n^{-1/4})$. We obtain the same rate of convergence in the Kolmogorov distance,
using $|T'-T| \leq \frac{\text{const.}}{n^{3/4}} =: A$. The order of the first two summands in \eqref{kolall} is $\mathcal{O}(n^{-1/4})$. 
The third term in \eqref{kolall} is of order $\mathcal{O}(n^{-3/4})$ and finally
$$
\frac{3A}{2} \E |\psi(T)| \leq \frac{\text{const}}{n^{3/4}} \E|T^3| = \mathcal{O}(n^{-3/4}),
$$
which completes the proof.
\end{proof}

\begin{proof}[Proof of Theorem \ref{TV}]
Since $V_1=0$, by the continuous mapping theorem it suffices to show that the random vector $\bar{V}:= (V_2,\ldots,V_q)$ 
converges towards the $(q-1)$-dimensional centered Gaussian vector with covariance matrix
\begin{center}
$\frac{q}{2(q-1)^2(q-2)}\begin{pmatrix} q-2 & -1 & ... & ... & ...&-1\\  -1 & 
q-2 & -1 & ... & ...&-1 \\  ... &... & ... & ...&...&...\\  -1 & ... & ... & ... & -1&q-2   \end{pmatrix}$.
\end{center}
We will apply Theorem \ref{RR}. We see for any $i \geq 2$
$$
\E [V_i'-V_i \mid \mathcal{F}]  = \E[W_i' - W_i \mid \mathcal{F}] = -\frac{1}{n\sqrt{n}}\sum\limits_{j=1}^n Y_{ji}+\frac{1}{n\sqrt{n}}\sum\limits_{j=1}^n\mathbb{E}[Y_{ji}'\mid \mathcal{F}].
$$
With Lemma \ref{expGD} and $R_n^{(1)}(i, h_0)$ defined as in \eqref{R1i} we obtain
\begin{eqnarray*}
\mathbb{E}[V_i'-V_i\mid \mathcal{F}]&=& -\frac{1}{n\sqrt{n}}\sum\limits_{j=1}^nY_{ji}+\frac{1}{\sqrt{n}} m_i(\sigma) -\frac{1}{\beta_0\sqrt{n}}
\frac{\partial}{\partial x_i}G_{\beta_0,h_0}(m(\sigma))+R_n^{(1)}(i, h_0) \\
&=& - \frac{1}{\beta_0\sqrt{n}}\frac{\partial}{\partial x_i}G_{\beta_0,h_0}\left(x+\frac{T\,u}{n^{1/4}}+\frac{V}{n^{1/2}}\right)+R_n^{(1)}(i,h_0).
\end{eqnarray*}
By the {\it fourth-order} Taylor expansion of  $\frac{\partial}{\partial x_i}G_{\beta_0,h_0}\left(x+\frac{T\,u}{n^{1/4}}+\frac{V}{n^{1/2}}\right)$,
see \eqref{arbeit1}, \eqref{arbeit2}, \eqref{arbeit3} and \eqref{finalzwei} in the Appendix,  we obtain for any $i \in \{2, \ldots, q\}$
\begin{equation} \label{forheuristic}
\mathbb{E}[V_i'-V_i\mid V]= - \frac{4}{\beta_0 \, n \, q^2} \langle (1, \ldots, 1, (q^2-3q+3), 1 \ldots, 1), \bar{V} \rangle + R_i,
\end{equation}
where $(q-1)(q-2)$ is the $i$'th entry of the vector in $\R^{q-1}$ and 
$$
R_i:= {\mathcal O} \bigl(\sum_{l=1}^q \frac{V_l \, T}{n^{5/4}} \bigr) + {\mathcal O} \biggl( \frac{1}{\sqrt{n}} A_2(i, V/\sqrt{n}, T/n^{1/4}) \biggr) + 
{\mathcal O} \bigl( \frac{T^3}{n^{5/4}} \bigr) + 
{\mathcal O} \bigl( \frac{T^4}{n^{5/2}} \bigr) + R_n^{(1)}(i,h_0)
%\mathcal{O}\left(\frac{T^2}{n}\right)+\mathcal{O}\left(\frac{V_2^2}{n\sqrt{n}}\right)+R_n^{(1)}(i,h_0).$
$$
(see \eqref{arbeit2} for the definition of $A_2$).
Using $\sum_{k=2, k \not= i}^q V_k = - V_i$ we get 
$$
\langle (1, \ldots, 1, (q^2-3q+3), 1 \ldots, 1), \bar{V} \rangle  = \bigl( (q-1)(q-2)\bigr) V_i.
$$
Thus the linearity condition of Theorem \ref{RR} is satisfied with $\Lambda= \frac 1n \frac{(q-2)}{q} \,\text{Id}_{q-1\times q-1}$ 
and $R=(R_2, \ldots, R_q)$. With  $q-2 >0$ for $q \geq 3$ we get the invertibility of $\Lambda$ and
$\lambda^{(i)}= \mathcal{O}(n)$. From the proof of Theorem \ref{THUM} we see that
$$
\mathbb{E}[(V_i'-V_i)(V_j'-V_j)\mid \mathcal{F}]=\mathbb{E}[(W_i'-W_i)(W_j'-W_j)\mid \mathcal{F}]=\mathcal{O}(n^{-3/2})
$$
and thus $A$ in Theorem \ref{RR} is of order $\mathcal{O}(n^{-1/2})$. Moreover $|V_i'-V_i| \leq \frac{1}{\sqrt{n}}$ for all $i$ and thus
$B$ in Theorem \ref{RR} is of order $\mathcal{O}(n^{-1/2})$. It remains to calculate $C$ in Theorem \ref{RR}. From the proof of Theorem \ref{THUM}
we know that $\E (R_n^{(1)}(i,h_0)) = {\mathcal O}(n^{-3/2})$. Using bounded moments of $V_i$ and $T$  we obtain that the expectations
of the first and third term of $R_i$ are ${\mathcal O}(n^{-5/4})$. With \eqref{arbeit2} we further get that the expectations of the
second and fourth term of $R_i$ are ${\mathcal O}(n^{-3/2})$ and ${\mathcal O}(n^{-5/2})$, respectively. By the Cauchy-Schwartz inequality 
we get $\sqrt{\mathbb{E}[R_i^2]}=\mathcal{O}(n^{-5/4})$, hence $C=\mathcal{O}(n^{-1/4})$, which completes the proof.
\end{proof}

\begin{remark}
In Remark \ref{heuristic} we gave a heuristic, that the matrix $\Lambda$ in the regression condition \eqref{regressioncond}
should be expected to be $\frac{1}{\beta n} D^2G_{\beta,h}(x)$. Our heuristic is confirmed in the proof
of Theorem \ref{TV}, since we can rewrite \eqref{forheuristic} as
$$
\mathbb{E}[\bar{V}'- \bar{V} \mid \bar{V}]= - \frac{1}{\beta_0 \, n} D^2 G_{q-1}(x) \, \bar{V} + R,
$$
where $D^2 G_{q-1}(x)$ denotes the upper $(q-1) \times (q-1)$ part of $D^2 G_{\beta_0, h_0}(x)$. 
The limiting covariance matrix $\Sigma$ of $\bar{V}$ is given by $[D^2 G_{q-1}(x)]^{-1} - \beta_0^{-1} \text{Id}_{q-1}$.
\end{remark}

\begin{remark}
These rates of convergence remain still valid if we change our probability measure $P_{ \beta,h,n}$ to
$$
P_{ \beta, h, n}( \sigma)=\frac{1}{Z_{ \beta,h,n}}\exp\left(\frac{ \beta}{2n}\sum\limits_{1\leq i\leq j\leq n}
\delta_{ \sigma_i,\sigma_j}+\sum\limits_{j=1}^n\sum\limits_{i=1}^q\delta_{ \sigma_j,i}h_i\right)
$$
for $\beta \in \mathbb{R}^+$ and $h \in \mathbb{R}^q$. This measure and the characteristics of the corresponding function
$$
G_{\beta,h}(u):=\frac{\beta}{2} \langle u,u \rangle - \log \left(\sum\limits_{i=1}^q\exp\left(\beta u_i+h_i\right)\right)
$$
were studied in \cite{Wang:1994}. First of all we note that \eqref{PBH} is the same model with $h_1=h$ and $h_i=0$ for $i\in \{2,\ldots,q\}$. 
Based on the results of \cite{Ellis/Wang:1990} and \cite{Wang:1994} and following the same procedures as above our results can 
easily be extended to the case that $h_i\neq 0$ for $i\in \{2,\ldots,q\}$. We omit these extensions here.
\end{remark}

%\newpage
\section{Appendix}

For the proofs of Theorems \ref{THUM}, \ref{THUM2} and \ref{THMM} and for Lemma \ref{wangimprovement} and Lemma \ref{EnMo} we need a multivariate 
{\it second-order Taylor-expansion} of $G_{\beta,h}$ defined in \eqref{Gfunktion}, for every $(\beta,h) \not= (\beta_0, h_0)$. 
Let us denote by $D^2G_{\beta,h}(x)$ the Hessian matrix
$\{ \partial^2 G_{\beta,h}(x) / \partial x_i \partial x_j, i,j = 1, \ldots, q\}$ of $G_{\beta,h}$ at $x$. We obtain
\begin{eqnarray} \label{taylor1}
G_{\beta,h}(u)& = &  G_{\beta,h}(x)+\sum\limits_{k=1}^q\frac{\partial}{\partial u_k}G_{\beta,h}(x)(u_k-x_{k})+
\frac{1}{2} \langle (u-x),D^2G_{\beta,h}(x)\cdot (u-x) \rangle \\ \nonumber
& + & \frac{1}{6}\sum\limits_{t,k,j=1}^q \widetilde R_{t,k,j}(u_t-x_{t})(u_k-x_{k})(u_j-x_{j})
\end{eqnarray}
with
$| \widetilde R_{t,k,j} | \leq \parallel \frac{\partial^3}{\partial u_k\partial u_t\partial u_j}G_{\beta,h} \parallel$.
For any fixed $m\in \{1,\ldots,q\}$ and any $x, u \in \mathbb{R}^q$ it follows that
\begin{eqnarray} \label{taylorsecond}
\frac{\partial}{\partial u_m}G_{\beta,h}(u) & = & \frac{\partial}{\partial u_m}G_{\beta,h}(x) + \frac{\partial^2}{\partial^2 u_m} G_{\beta,h}(x)(u_m-x_{m})\\
& + & \sum\limits_{k\neq m}\frac{\partial^2}{\partial u_k\partial u_m}G_{\beta,h}(x)(u_k-x_{k}) + \sum\limits_{k=1}^q\mathcal{O}((u_m-x_{m})(u_k-x_{k})) \nonumber \\
& +& \sum\limits_{k,t \neq m}^q\mathcal{O}((u_k-x_{k})(u_t-x_{t})). \nonumber
\end{eqnarray}

If $x$ is a global minimum point of $G_{\beta,h}$ we are able to calculate the Hessian as follows.

\begin{lemma} \label{hessiancal}
The Hessian $D^2 G_{\beta,h}(x_0)$ at an arbitrary global minimum point $x_0$ looks like
$$
\frac{\partial^2}{\partial^2 u_1} G_{\beta,h}  (x_0) = \beta - \beta^2 (q-1) x_{0,1} \, x_{0,q}, \quad
\frac{\partial^2}{\partial u_1\partial u_q} G_{\beta,h} (x_0) = \beta^2 x_{0,1}\, x_{0,q},
$$ 
and
$$
\frac{\partial^2}{\partial^2 u_q} G_{\beta,h} (x_0) = \beta - \beta^2 (x_{0,1} \, x_{0,q} + (q-2) x_{0,q}^2), \quad
\frac{\partial^2}{\partial u_2 \partial u_q}G_{\beta,h} (x_0) = \beta^2 x_{0,q}^2.
$$
\end{lemma}

\begin{proof}
According to Proposition \ref{ChMi} we know that for any minimizer $x_0$ of $G_{\beta,h}$ we have
$x_{0,i} = x_{0,k}$ for all $i,k \in \{2,\ldots,q\}$ and  $x_{0,1} \geq x_{0,k}$  for all $k\in\{2,\ldots,q\}$
and $\sum\limits_{i=1}^qx_{0,i}=1$. Notice that the equation $\nabla G_{\beta,0}(x_0)=0$ implies 
\begin{align}
x_{0,1}&=\frac{\exp(\beta x_{0,1}+ h \delta_{1,1})}{\exp(\beta x_{0,1} +h )+(q-1)\exp(\beta x_{0,q})},\nonumber\\
x_{0,q}&=\frac{\exp(\beta x_{0,q})}{\exp(\beta x_{0,1}+ h )+(q-1)\exp(\beta x_{0,q})}.\nonumber
\end{align} 
Now we can calculate
\begin{eqnarray*}
\frac{\partial^2}{\partial^2 u_1} G_{\beta,h}(x_0) 
& = & \beta - \beta^2 (q-1) \frac{\exp(\beta (x_{0,1} + x_{0,q}) + h)}{(\exp(\beta (x_{0,1}) + h) +(q-1) \exp(\beta x_{0,q}))^2} \\
& = & \beta - \beta^2 (q-1) x_{0,1} \, x_{0,q}
\end{eqnarray*}
and
$$
\frac{\partial^2}{\partial u_1\partial u_q} G_{\beta,h} (x_0) = 
\beta^2  \frac{\exp(\beta (x_{0,1}+x_{0,q})+h)}{(\exp(\beta x_{0,1}+h)+(q-1)\exp(\beta x_{0,q}))^2}= \beta^2  \, x_{0,1} \, x_{0,q}.
$$
Moreover
\begin{eqnarray*}
\frac{\partial^2}{\partial^2 u_q} G_{\beta,h} (x_0) & = & \beta -\beta^2
\frac{\exp(\beta x_{0,q} +h)(\exp(\beta x_{0,1}+h)+(q-2)\exp(\beta x_{0,q}))}{(\exp(\beta x_{0,1}+h)+(q-1)\exp(\beta x_{0,q}))^2} \\
& =& \beta-\beta^2 (x_{0,1}\, x_{0,q}+(q-2) x_{0,q}^2)
\end{eqnarray*}
and
$$
\frac{\partial^2}{\partial u_2 \partial u_q}G_{\beta,h} (x_0)=
\beta^2 \frac{\exp(2\beta x_{0,q})}{(\exp(\beta x_{0,1}+h)+(q-1)\exp(\beta x_{0,q}))^2}= \beta^2 \, x_{0,q}^2.
$$
\end{proof}

With Lemma \ref{hessiancal} we get, that the Hessian of $G_{\beta,h}$ at a global minimum point is a matrix of type \eqref{Matrix}.
The following Lemma is some Linear Algebra for a matrix of the form \eqref{Matrix}.

\begin{lemma} \label{DetLe}
For any $a,b,c,d\in\mathbb{R}$ consider the following matrix
\begin{equation} 
\Lambda:=\begin{pmatrix} a & b & ... & ... & ... & ...&b \\ b & d & c & ... & ... & ...&c\\ b & c & d & c & ... & ...&c 
\\ ... & ... &... & ... & ...&...&...\\ b & c & ... & ... & ... & c&d   \end{pmatrix}\in\mathbb{R}^{q\times q}. \label{Matrix}
\end{equation}
Then $\det(\Lambda)=(d-c)^{q-2}(a(d+(q-2)c)-(q-1)b^2)$. 
\end{lemma}

\begin{proof} 
We applied the formula due to Laplace.
\end{proof}

\begin{remark} \label{extrem}
At the extremity $(\beta_0,h_0)= (4 \frac{q-1}{q}, \log(q-1) - 2 \frac{q-2}{q})$ of the critical line, 
$x=(1/2, 1/2(q-1), \ldots, 1/2(q-1))$ is the unique minimum point of $G_{\beta_0,h_0}$. 
Notice that
$$
\exp(\beta_0\cdot x_1+h_0) = (q-1) \exp(2/q), \quad \exp(\beta_0\cdot x_q) = \exp(2/q).
$$
With Lemma \ref{hessiancal} we obtain 
$$
\frac{\partial^2 G_{\beta_0,h_0}}{\partial^2 x_1}(x) = \frac{\partial^2 G_{\beta_0,h_0}}{\partial x_1\partial x_q}(x) = \frac{4(q-1)}{q^2}
$$
and $\frac{\partial^2 G_{\beta_0,h_0}}{\partial^2 x_q}(x)=\frac{4(q^2-3q+3)}{q^2}$ and
$\frac{\partial^2 G_{\beta_0,h_0}}{\partial x_q\partial x_{q-1}}(x)=\frac{4}{q^2}$. Thus $a=b$ in \eqref{Matrix}.
\end{remark}

For the proofs of the results at {\it criticality}, Theorems \ref{TT}, \ref{TV} and Lemma \ref{EnMo2}, we need a multivariate 
{\it fourth-order Taylor-expansion} of $G_{\beta_0,h_0}$ (defined in \eqref{Gfunktion}). We fix the notation $G := G_{\beta_0,h_0}$
for the following calculations. We know that $x=(1/2, 1/2(q-1), \ldots, 1/2(q-1))$ is the unique
minimum point of $G$. Let $u=(1-q, 1, \ldots, 1) \in {\mathcal M} \subset \R^q$, $v \in {\mathcal M} \cap u^{\bot}$ and $t \in \R$.
For any $p \in \N$ and $z \in \R^q$ let us fix the notation
$$
R_{i_1, \ldots, i_p} (z) := \bigl( \frac{\partial^p G}{\partial x_{i_1} \cdots \partial x_{i_p}} \bigr) (z).
$$
A {\it second-order} Taylor-expansion yields
$$
G(x+tu+v) = G(x+tu) + \frac 12 \langle v, (D^2G)(x+tu) \, v \rangle + \frac 16 \sum_{j,k,l=1}^q R_{j,k,l}(x+tu+\gamma v) \, v_j v_k v_l
$$
for some $\gamma \in (0,1)$, since  $\langle (\nabla G)(x+tu), v \rangle=0$, the last $q-1$ coordinates of $x+tu$ are equal and with Lemma \ref{expGD} 
the last $q-1$ coordinates of the gradient $(\nabla G)(x+tu)$ are equal, and hence it is orthogonal to $v$. 
A {\it fourth-order} Taylor-expansion for $G(x+tu)$ yields
\begin{equation} \label{huebsch}
G(x+tu) = G(x) + \frac{1}{24} \sum_{j,k,l,m=1}^q R_{j,k,l,m}(x+ \widetilde{\gamma} t u) \, t^4 \, u_j u_k u_l u_m
%+ \frac{1}{5!} \sum_{i_1, \ldots, i_5=1}^q R_{i_1, \ldots, i_5}(x) \, t^5 \, u_{i_1} \cdots u_{i_5}.
\end{equation}
for some $\widetilde{\gamma} \in (0,1)$. To see \eqref{huebsch} notice that the first-order term is zero since $x$ is a global minimizer of $G$. 
The second-order term is zero, since we know from Lemma 
\ref{wangimprovement} that $D^2G_{\beta,h}(x)$ is positive definite if and only if $(\beta,h) \not= (\beta_0,h_0)$. Hence the third-order term is zero yielding
the identity \eqref{huebsch}. Summarizing we obtain
\begin{eqnarray} \label{4thorder}
G(x+tu+v) & = &  G(x) + \frac 12 \langle v, (D^2G)(x+tu) \, v \rangle \nonumber +  
\frac 16 \sum_{j,k,l=1}^q R_{j,k,l}(x+tu+ \gamma v) \, v_j v_k v_l \\ &+&  \frac{1}{24} \sum_{j,k,l,m=1}^q R_{j,k,l,m}(x+ \widetilde{\gamma} tu) \, t^4 \, u_j u_k u_l u_m.
%& + &  \frac 16 \sum_{j,k,l=1}^q R_{jkl} \, v_t v_k v_j + \frac{1}{5!} \sum_{i_1, \ldots, i_5=1}^q R_{i_1, \ldots, i_5}(x) \, t^5 \, u_{i_1} \cdots u_{i_5}.
\end{eqnarray}
With $y := x+tu+v$ we will calculate $\frac{\partial}{\partial y_i} G(y)$ for $i \in \{1, \ldots, q\}$ using \eqref{4thorder}. The derivative of the 
first summand in \eqref{4thorder} is zero since $x$ is the global minimizer of $G$. With
\begin{eqnarray*}
\frac 12 \langle v, (D^2G)(x+tu) \, v \rangle & = &  \frac 12 R_{i,i}(x+tu)(y_i-x_i-t u_i)^2 \\
& + &  \sum_{k \not=i} R_{i,k}(x+tu)(y_k-x_k-tu_k)(y_i-x_i-tu_i) \\
& + &  \sum_{k,l \not=i} R_{l,k}(x+tu)(y_k-x_k-tu_k)(y_l-x_l-tu_l)
\end{eqnarray*}
we obtain
$$
A_1(i) := \frac{\partial}{\partial y_i} \bigl( \frac 12 \langle v, (D^2G)(x+tu) \, v \rangle \bigr) = R_{i,i}(x+tu) v_i + \sum_{k \not= i} R_{i,k}(x+tu) v_k.
$$
With Lemma \ref{hessiancal} we obtain $R_{1,2} = R_{1,3}=\cdots =R_{1,q}$ and since $v \in u^{\bot}$ we have $A_1(1)=0$.
With a second-order Taylor expansion for $R_{i,k}(x+tu)$ we get for $t$ small
\begin{equation*} 
A_1(i) = \langle R_{i, \cdot}, v \rangle + {\mathcal O}_q \bigl( \sum_{l=1}^q v_l \, t \bigr).
\end{equation*}
Here ${\mathcal O}_q$ is the notation that the constant does depend on $q$. This is because all $R_{i_1, \ldots, i_p}(x)$ only depend on $q$, since
$x$ only depends on $q$. The second partial derivatives of $G$ were listed in Remark \ref{extrem}, hence we end up with
\begin{equation} \label{arbeit1}
A_1(1)=0, \quad A_1(i) = \frac{4(q^3-3q+3)}{q^2} v_i + \frac{4}{q^2} \sum_{k=2}^q v_k + {\mathcal O}_q \bigl( \sum_{l=1}^q v_l \, t \bigr), \,\, i \geq 2
\end{equation}
for small $t$. The last formula can even be simplified since $\sum_{k=2}^q v_k =0$ using $v \in u^{\bot}$. For reasons of application we will not use this simplification.
The partial derivative $\frac{\partial}{\partial y_i}$ of the third term in \eqref{4thorder} is
$$
A_2(i) := \frac 12 R_{i,i,i}(x+tu+\gamma v) v_i^2 + \sum_{k \not=i}^q R_{i,i,k}(x+tu+\gamma v) \, v_i v_k + \frac 12 \sum_{j,k \not= i} R_{i,j,k}(x+tu + \gamma v) \, v_k v_j.
$$
Using Taylor for $R_{i,j,k}(x+tu + \gamma v)$ we obtain for small $t$ and small $v$ 
\begin{equation} \label{arbeit2}
A_2(i) := A_2(i,v,t) := \frac 12 R_{i,i,i}(x) v_i^2 + \sum_{k \not=i}^q R_{i,i,k}(x) \, v_i v_k + \frac 12 \sum_{j,k \not= i}^q R_{i,j,k}(x) \, v_k v_j + {\mathcal O}_q( f(v,t))
\end{equation}
with 
$$
{\mathcal O}_q (f(v,t)) = {\mathcal O}_q \biggl( (t + \sum_{k =2}^q c_k(q) v_k) \bigl( v_i^2 +  v_i \sum_{k \not= i} v_k + \sum_{k,j \not= i} v_k v_j \bigr) \biggr). 
$$
Here $c_k(q)$ is denoted (a constant depending on $q$) just to see that the relation $\sum_{k=2}^q v_k=0$ cannot be applied. 
In our application the order (in $n$) of the $v_k$'s will not depend on $k$, and $v_2$ will be smaller in order than $t$.
In this situation we have ${\mathcal O}_q (f(v,t)) =  {\mathcal O}_q (t \, v_2^2)$.
We will calculate $A_2(1)$ with the help of the third derivatives which are $R_{1,1,1}(x) = R_{1,1,k}(x)=0$
and
$$
R_{1,j,j}(x) =\frac{16(q-1)(q-2)}{q^3}, \quad R_{1,j,k}(x)= -\frac{16(q-1)}{q^3}.
$$
Therefore the first two summands in \eqref{arbeit2} are zero and the third term is, using $\sum_{k=2}^q v_k=0$, 
\begin{eqnarray} \label{arbeit2fall1}
\frac 12 \sum_{j,k=2}^q R_{1,j,k}(x) v_j v_k & =& \frac 12 \sum_{j=2}^q \frac{16(q-1)(q-2)}{q^3} \, v_j^2 - \frac{1}{2}\sum\limits_{j,k=2, j\neq k}^q
\frac{16(q-1)}{q^3}\, v_j v_k \nonumber \\
&=& \frac{8(q-1)}{q^3}\sum\limits_{j=2}^q v_j \bigl( (q-1)v_j - v_j - \sum\limits_{k=2, k\neq j}^q v_k\bigr) \\
&=& \frac{8(q-1)^2}{q^3} \sum\limits_{j=2}^q v_j^2. \nonumber
\end{eqnarray}
Finally, the partial derivative $\frac{\partial}{\partial y_i}$ of the fourth term in \eqref{4thorder} is
\begin{eqnarray} \label{arbeit3}
A_3(i) := A_3(i,t) & := & \frac 16 R_{i,i,i,i}(x) t^3 u_i^3 + \frac 12 \sum_{k \not= i} R_{i,i,i,k}(x) t^3 u_i^2 u_k + \frac 12 \sum_{k,j \not= i} R_{i,i,j,k}(x) t^3 u_i u_k u_j \nonumber \\
& + & \frac 16 \sum_{j,k,l \not= i} R_{i,j,k,l}(x) t^3 u_j u_k u_l + {\mathcal O}_q(t^4),
\end{eqnarray}
where we applied a second-order Taylor expansion for $ R_{i,j,k,l}(x + \widetilde{\gamma} t u)$. Again we calculate $A_3(1)$, using
the fourth derivatives
$$
R_{1,1,1,1}(x) = \frac{32(q-1)^4}{q^4}, R_{1,1,1,k}(x)= \frac{32(q-1)^3}{q^4},\quad
R_{1,1,k,k}(x) = R_{1,1,j,k}(x)=  \frac{32(q-1)^2}{q^4}
$$
and 
$$
R_{1,k,k,k}(x) =\frac{32(q-1)(2q^2-10q+11)}{q^4}, \quad R_{1,j,j,k}(x) =-\frac{32(q-1)(2q-5)}{q^4}
$$
and $R_{1,j,k,l} =\frac{96(q-1)}{q^4}$.
We obtain
$$
\frac 16 R_{1,1,1,1}(x) t^3 (1-q)^3 = - \frac{16(q-1)^7 t^3}{3 q^4}, \quad \frac 12 \sum_{k \not= i} R_{1,1,1,k}(x) t^3 (1-q)^2  = - \frac{16(q-1)^6 t^3}{q^4},
$$
and 
$$
\frac 12 \sum_{k,j \not= i} R_{1,1,j,k}(x) t^3 (1-q)  = - \frac{16(q-1)^5 t^3}{q^4}.
$$
Moreover we have
\begin{eqnarray*}
\frac 16 \sum_{j,k,l \not= i} R_{1,j,k,l}(x) t^3 & = & \frac{16(q-1)^2(2q^2-10q+11) t^3}{3 q^4} - \frac{16 (q-1)^2(q-2)(2q-5) t^3}{q^4} \\
& & \hspace{1cm} + \frac{16(q-1)^2(q-2)(q-3) t^3}{q^4}.
\end{eqnarray*}
Hence
\begin{equation} \label{arbeit3fall1}
A_3(1) = - \frac{16(q-1)^4}{3q} t^3 + {\mathcal O}_q(t^4).
\end{equation}
We summarize that the first partial derivative  of $G(y)$ in \eqref{4thorder} satisfies
\begin{equation} \label{finaleins}
\frac{\partial}{\partial y_1} G (x+tu+v) = \frac{8(q-1)^2}{q^3} \sum\limits_{j=2}^q v_j^2  - \frac{16(q-1)^4}{3q} t^3 +{\mathcal O}_q\bigl( f(v,t) \bigr) 
+ {\mathcal O}_q(t^4),
\end{equation}
using the notation of \eqref{arbeit2}. The $i$'th partial derivative for $i \in \{2, \ldots,q \}$ is given by
\begin{equation} \label{finalzwei}
\frac{\partial}{\partial y_i} G (x+tu+v) = A_1(i) + A_2(i,v,t) + A_3(i,t)
\end{equation}
with $A_j(i)$ defined in \eqref{arbeit1}, \eqref{arbeit2} and \eqref{arbeit3}.

\begin{proof}[Proof of Lemma \ref{wangimprovement}]
For the proof we will use the following alternative parametrization of the minimum points of $G_{\beta,h}$ given by permutations of
$$
x_s=\left(\frac{1+(q-1)s(\beta,h)}{q},\frac{1-s(\beta,h)}{q},\cdots,\frac{1-s(\beta,h)}{q}\right), \quad s(\beta,h)\in [0,1].
$$
It is important to notice, see for example \cite{Blanchard:2008}, that $s(\beta,h)$ is positive, well-defined and strictly increasing 
in $\beta$ on an open interval containing $[\beta_c,\infty)$ and that $s(\beta_c,0)=(q-2)/(q-1)$ is a global minimum.
The Lemma follows once we proof that $\beta q x_{s,1} x_{s,q}-1<0$ and $\beta x_{s,q}-1<0$. The second inequality follows directly from 
Proposition \ref{ChMi}. To prove  $\beta q x_{s,1} x_{s,q}-1<0$, we first consider the case $h=h_0$.
First of all we note that the minima are the solutions of
$$
f(s(\beta,h_0))=\log(1+(q-1)s(\beta,h_0))-\log(1-s(\beta,h_0))-\beta s(\beta,h_0)-h_0=0.
$$
If $\beta<\beta_0$ we have that $\partial f(s(\beta,h_0))/\partial s>0$. Rearranging this equality and using the parametrization of $x_s$ yields the result.
If $\beta>\beta_0$ we use $\nabla G_{\beta,h_0}(x_s)=0$ to obtain
$$
\log(x_{s,1})-\beta x_{s,1}-h_0=\log(x_{s,q})-\beta x_{s,q}.
$$
This equation yield %and the fact that $x_{s,1}+(q-1)x_{s,q}=1$ yield
$$
\beta q x_{s,1}x_{s,q}-1 =\left(\frac{\log(x_{s,1})-\log(x_{s,q})-h_0}{x_{s,1}-x_{s,q}}\right)qx_{s,1}x_{s,q}-1
=\left(\log\left(\frac{x_{s,1}}{x_{s,q}}\right)-h_0\right)q\frac{x_{s,1}x_{s,q}}{x_{s,1}-x_{s,q}}-1.
$$
Using the fact that $x_{s,1}+(q-1)x_{s,q}=1$ we obtain
\begin{eqnarray*}
\beta q x_{s,1}x_{s,q}-1 &= & \left(\log\left(\frac{(q-1)x_{s,1}}{1-x_{s,1}}\right)-h_0\right)q\frac{x_{s,1}(1-x_{s,1})}{(q-1)x_{s,1}-(1-x_{s,1})}-1\nonumber\\\nonumber\\
&=&\left(\log\left(\frac{(q-1)x_{s,1}}{1-x_{s,1}}\right)-h_0\right)q\frac{x_{s,1}(1-x_{s,1})}{qx_{s,1}-1}-1 =q\frac{x_{s,1}(1-x_{s,1})}{qx_{s,1}-1}f_{h_0}(x_{s,1}),\nonumber
\end{eqnarray*}
where
$$
f_h(x_{s,1}):=\log\left(\frac{(q-1)x_{s,1}}{1-x_{s,1}}\right)-h-\frac{qx_{s,1}-1}{qx_{s,1}(1-x_{s,1})}.
$$
For $\beta_0$ the global minimizer is given by $y=(1/2,1/2(q-1),\ldots,1/2(q-1))$ and it follows easily that
$f_{h_0}\left(\frac{1}{2}\right)=0$.
Additionally $\frac{\partial f_{h_0}}{\partial x}(x)<0$ for $x \in[2^{-1},1)$.
Thus we obtain that $f_{h_0}(x_{s,1})<0$ for $x_{s,1}\in[2^{-1},1)$, and this is equivalent to $\beta q x_{s,1}x_{s,q}-1<0$.

Now we consider the case $h\neq h_0$. If $\beta<\beta_0$, the proof is identical to the case of $\beta<\beta_0$ for $h=h_0$.
If $\beta\geq\beta_0$ we have
$$
\beta q x_{s,1}x_{s,q}-1=q\frac{x_{s,1}(1-x_{s,1})}{qx_{s,1}-1}f_{h}(x_{s,1}).
$$
For $\beta\geq\beta_0$ we know that $1>s(\beta,h)\geq s(\beta_0,h)\geq s(\beta_c,0)=(q-2)/(q-1)$.
The rest of the proof is identical to the case (ii) of the proof of Proposition 2.2 in \cite{Ellis/Wang:1990}. 
\end{proof}

\noindent
{\bf Acknowlegement:} We would like to thank two anonymous referees for helpful comments.

%\newcommand{\SortNoop}[1]{}\def\cprime{$'$}
%\providecommand{\bysame}{\leavevmode\hbox to3em{\hrulefill}\thinspace}
%\providecommand{\MR}{\relax\ifhmode\unskip\space\fi MR }

\providecommand{\MRhref}[2]{%
  \href{http://www.ams.org/mathscinet-getitem?mr=#1}{#2}
}
\providecommand{\href}[2]{#2}
%\begin{thebibliography}{10}
%\end{thebibliography}

%\bibliographystyle{amsplain}
%\bibliography{31.08.2012}

%\bibliographystyle{amsplain}
%\bibliography{15.09.2010}

\newcommand{\SortNoop}[1]{}\def\cprime{$'$} \def\cprime{$'$}
  \def\polhk#1{\setbox0=\hbox{#1}{\ooalign{\hidewidth
  \lower1.5ex\hbox{`}\hidewidth\crcr\unhbox0}}}
\providecommand{\bysame}{\leavevmode\hbox to3em{\hrulefill}\thinspace}
\providecommand{\MR}{\relax\ifhmode\unskip\space\fi MR }
\providecommand{\MRhref}[2]{%
  \href{http://www.ams.org/mathscinet-getitem?mr=#1}{#2}
}
\providecommand{\href}[2]{#2}

\end{document}